%% file: kronecker-lattices.tex
\documentclass{article}

\usepackage{pdfsync}
\input{macros}

\usepackage{amsmath}
\usepackage{hyperref}

\makeatletter
\newcommand{\doublewidetilde}[1]{{%
  \mathpalette\double@widetilde{#1}%
}}
\newcommand{\double@widetilde}[2]{%
  \sbox\z@{$\m@th#1\widetilde{#2}$}%
  \ht\z@=.9\ht\z@
  \widetilde{\box\z@}%
}
\makeatother

\newcommand{\ralpha}{r_{\alpha, \bsgamma}}
\newcommand{\eworKor}[1]{e_{#1,\alpha,\bsgamma}^{\mathrm{wor}}}
\newcommand{\ewor}{e_{\Hsigma}^{\mathrm{wor}}}
\newcommand{\KorSpace}{{\calH_{d,\alpha,\bsgamma}^\textup{Kor}}}
\newcommand{\Hsigma}{{\calH_\bssigma}}
\newcommand{\phii}{{\phi_i}}
\newcommand{\normKor}[1]{{\left\|#1\right\|_{d,\alpha,\bsgamma}}}
\newcommand{\muquant}{\mu}

\newcommand{\Pn}{{P_n}}
\newcommand{\Pnzeta}{{P_n^\bszeta}}
\newcommand{\PnNz}{{P_n^{\bsz/N}}}
\newcommand{\LS}{{\, \textup{LS}}}
\newcommand{\lsq}{{\, \textup{LS}_{m}^{\Pn}}}
\newcommand{\lsqzeta}{{\, \textup{LS}_{m}^{\Pnzeta}}}
\newcommand{\lsqNz}{{\, \textup{LS}_{m}^{\PnNz}}}
\newcommand{\Phinm}{\Phi_{\Pn, m}}
\newcommand{\Phinzetam}{\Phi_{\Pnzeta, m}}
\newcommand{\Phinzm}{\Phi_{\PnNz, m}}
\newcommand{\Phininf}{\Phi_{\Pn, \infty}}
\newcommand{\Phinzetainf}{\Phi_{\Pnzeta, \infty}}
\newcommand{\PhinNzinf}{\Phi_{\PnNz\!\!, \infty}}
\newcommand{\PhinNzm}{\Phi_{\PnNz\!, m}}
\newcommand{\Phim}{\Phi_{m}}
\newcommand{\Rminf}{R_{m, \infty}}
\newcommand{\diag}{{\textup{diag}}}
\newcommand{\Iinf}{{I_{\infty \times \infty}}}
\newcommand{\Imm}{{I_{m \times m}}}
\newcommand{\Dlambda}{D_\bssigma}
\newcommand{\Dlambdam}{D_{\bssigma}'}
\newcommand{\Npos}{\N}
\newcommand{\Cinfty}{C_1}
\newcommand{\Chyp}{C_d}

\definecolor{darkred}{RGB}{139,0,0}
\definecolor{darkgreen}{RGB}{0,100,0}
\definecolor{darkmagenta}{RGB}{170,0,120}
\definecolor{darkpurple}{RGB}{110,0,180}
\definecolor{darkblue}{RGB}{40,0,200}
\definecolor{darkbrown}{rgb}{0.75,0.40,0.15}

\begin{document}

\title{Error bounds for function approximation using generated sets}

\author{Ronald Cools, Dirk Nuyens and Laurence Wilkes}

\date{\today}

\maketitle

\begin{abstract}
    This paper explores the use of ``generated sets'' $\{ \{ k \bszeta \} \st k = 1, \ldots, n \}$ for function approximation in reproducing kernel Hilbert spaces which consist of multi-dimensional functions with an absolutely convergent Fourier series.
    The algorithm is a least squares algorithm that samples the function at the points of a generated set.
    We show that there exist $\bszeta \in [0,1]^d$ for which the worst-case $L_2$ error has the optimal order of convergence if the space has polynomially converging approximation numbers.
    In fact, this holds for a significant portion of the generators.
    Additionally we show that a restriction to rational generators is possible with a slight increase of the bound.
    Furthermore, we specialise the results to the weighted Korobov space, where we derive a bound applicable to low values of sample points, and state tractability results.
\end{abstract}

\section{Introduction}\label{sec:introduction}

In this paper we analyse the use of generated sets in approximating $d$-dimensional functions with an absolutely convergent Fourier series.
For $n \in \Npos := \{ 1, 2, \ldots \}$ sample points, and a \emph{generator} $\bszeta \in [0, 1]^d$, the \emph{generated set} is the set of points
\begin{align}\label{eq:generated-set}
    \Pnzeta := \{ \{ k \bszeta \} \st k = 1, \ldots, n \}
    ,
\end{align}
where~$\{ \,\cdot\, \}$ represents the elementwise fractional part, equivalent to the elementwise modulo one operation.
These sets first appear in~\cite{K2013, K2014}, where they are defined to include a point at the origin.

Generated sets resemble \emph{Kronecker sequences}, which are infinite sequences of similarly defined points, but where $\bszeta$ is chosen to have exclusively irrational components.
Kronecker demonstrated that sequences of this form are dense in the $d$-dimensional unit cube $[0,1]^d$ and later Weyl~\cite{W1916} showed that the points in the sequence are uniformly distributed.
The discrepancy properties of these sequences have also been analysed and applied to the problem of numerical integration, see~\cite{KN1974, O1982, L1988, CVW2007}.
The points used in this paper deviate from Kronecker sequences as we will instead demonstrate that there is a high probability that a vector $\bszeta$ which has been uniformly randomly sampled from the unit cube, i.e., not necessarily having irrational components, will be a good candidate for the generator of a generated set.
Of course, this does not necessarily imply that its components would not all be irrational as, due to Cantor, irrational numbers are uncountable and dense in the real numbers.
However, we will show that the set of generators which are good must contain vectors with all rational components.
In the case of $\bszeta$ having all rational components, the point set resembles a \emph{rank-$1$ lattice}, see, e.g., \cite{KKKNS2022,CKNS2020,KMN2024,KPV2015}, however, if the common denominator would be $N$, we only take $n$ points, $n \le N$, and exclude the point at the origin for $k = 0$.

The function space we work with is a separable reproducing kernel Hilbert space, $\Hsigma$, defined by a non-negative square-summable sequence $\bssigma = (\sigma_\bsh)_{\bsh \in \Z^d}$ which dictates the relative importance of the Fourier coefficients $\hat{f}_\bsh$ at the indices $\bsh \in \Z^d$.
This space is
\[
    \Hsigma := \bigg\{ f \in L_2([0, 1]^d) \st \| f \|_\Hsigma^2 := \sum_{\bsh \in \Z^d} |\hat{f}_\bsh|^2 \, \sigma_\bsh^{-2} < \infty \bigg\}
    .
\]

The least squares algorithm $\lsq$ we employ approximates $m$ Fourier coefficients of a function~$f \in \Hsigma$ by using function values sampled at~$n \ge m$ points in the set $\Pn \subset [0, 1]^d$, where the $m$ Fourier coefficients correspond to the largest values of $\sigma_\bsh$ for $\bsh \in \Z^d$.
In our case, the point set we use will be a generated set $\Pnzeta$.
We will later equate~$\bssigma$ with the ordered sequence $(\sigma_i)_{i=1}^\infty$ which is the same sequence with some order such that it is non-increasing.

Such algorithms have appeared frequently in recent years, and \cite{KU2021a} demonstrates that it can produce a similar rate of convergence in separable reproducing kernel Hilbert spaces with square-summable linear widths to arbitrary linear algorithms for a significant proportion of the possible point sets which are formed by independently randomly sampling each point on the cube.
Further developments have been made to this asymptotic bound: in~\cite{KU2021b}, this is extended to the case of Banach spaces and in~\cite{KUV2021} the bound was improved by a logarithmic factor.
In~\cite{DKU2023}, it was demonstrated that there exist point sets for which separable reproducing kernel Hilbert spaces with finite trace attain a bound which is not only optimal in terms of the order but is also asymptotically sharp.
Here, we use order to refer to the polynomial order as defined in \cite{KU2021a}.
When applied to the function space $\Hsigma$, \cite[Equation~$(10)$]{DKU2023} implies that, for any $m \in \N$, there exists a point set $\Pn$ of size $n$ which, when used with the least squares algorithm, produces the following bound
\[
    \ewor(\lsq)
    \lesssim
    \sigma_{m+1} + \sqrt{\frac{1}{m} \sum_{i > m} \sigma_i^2}
    ,
\]
where $n \le c m$ for some universal constant $c \ge 1$ and $\ewor$ is the worst-case $L_2$ error for functions $f \in \Hsigma$ with $\| f \|_\Hsigma \le 1$.
For further notation, we will use $\lesssim$ to mean that the left hand side is eventually smaller than a constant times the right hand side, and any relevant dependencies of the constant will be explicitly mentioned.
There have also been results which utilise this technique for point sets which are structured.
In~\cite{BKPU2024}, the authors use a two-step subsampling procedure on a rank-$1$ lattice point set in order to produce a smaller point set which retains a good error decay.

The bound we prove in this paper does not achieve exactly the same bound as above but does achieve the same order of convergence in spaces where the sequence $\bssigma$ converges at least polynomially in $i \in \Npos$.
More precisely, in Theorem~\ref{thm:upperbound} we prove, for the space $\Hsigma$ under a mild additional regularity condition depending on some $r \in \Npos$, and for $m$ allowed to scale asymptotically with $n$ (depending on $\epsilon$), that the least squares error satisfies
\[
    \ewor(\lsqzeta)
    \le
    \sigma_{m+1} + \sqrt{\frac{1}{m}\sum_{i > m} \sigma_i^2 + \sqrt{\frac{1}{m^{1+r\epsilon}} \sum_{i > m} i^{r\epsilon} \, \sigma_i^4}}
    \qquad
    \text{for all }
    0 < \epsilon \le 1
    ,
\]
for a large proportion of generators $\bszeta \in [0,1]^d$.
Theorem~\ref{thm:upperbound} also contains a more general bound which does not depend on this regularity condition.

We note that instead of using the least squares algorithm, one can also use the kernel method, see, e.g., \cite{TW1980, KKKNS2022}, which is known to be the optimal algorithm, given the point set, and as such the worst-case $L_2$ error of using the generated set in combination with the kernel method is upper bounded by the above bound.
In Corollary~\ref{cor:sobolevzetabound}, we show that, for spaces with polynomial convergence of the sequence $\bssigma$ and a weak condition on the regularity of $\bssigma$ dependent on some $r \in \Npos$, the bound provided by Theorem~\ref{thm:upperbound} will become
\[
    \ewor(\lsqzeta)
    \lesssim
    n^{-\frac{1-\epsilon}{1+r\epsilon} \alpha} \log^{\alpha(d-1)}(n)
    \qquad
    \text{for all }
    0 < \epsilon \le 1
    ,
\]
where the implied constant tends towards infinity as $\epsilon$ approaches~$0$.
This bound has order $\alpha$, with $\epsilon$ taken arbitrarily close to~$0$.
This is known to be the optimal order, see, e.g., \cite{DTU2018, NW2008, NW2010, NW2012}.

Generated sets are described in~\cite{K2013}, where they are defined to include a point at the origin.
The paper~\cite{K2013} details how these sets can be used for function approximation and reconstruction with the same algorithm used here but for trigonometric polynomials, meaning that the functions have finite Fourier series expansions on a fixed index set.
In~\cite{K2013}, the ``rank-$1$ structure'', i.e., meaning ``multiples of one vector'' as in~\eqref{eq:generated-set}, is utilised to evaluate trigonometric polynomials at the points of a generated set using non-equispaced fast Fourier transforms.
This is expanded upon in~\cite{K2014}, which provides a more detailed description of the reconstruction algorithm's implementation.
It is shown that using an iterative method to solve the least squares problem, given by~\eqref{eq:LSmin} in Section~\ref{sec:leastsquaresalg}, leads to fast convergence when the condition number of the Fourier matrix is low and provides a bound on the condition number of the Fourier matrix which we refer to as~$\Phinm$ in Section~\ref{sec:leastsquaresalg}.
The condition number of $\Phinm$ therefore links not only to the stability of the problem but also to the computational complexity.
The PhD thesis~\cite{K2014}, also provides a method to search for reconstructing generated sets with low condition number of the Fourier matrix and demonstrates an upper bound for the worst-case $L_2$ error in the function space
\[
    \Big\{ f \in L_1([0, 1]^d) \st \sum_{\bsh \in \Z^d} \sigma_\bsh \, |\hat{f}_\bsh| < \infty \Big\}
    ,
\]
where $\sigma_\bsh \ge 1$ for all $\bsh \in \Z^d$.
This space is a subspace of the Wiener algebra, as the norm's definition implies absolute convergence of the Fourier series, and is also a subspace of $\Hsigma$ when the sequence $\bssigma$ is square summable.
For results in similar but more specific spaces related to the Korobov space, see~\cite{KPV2015}.

The layout of the paper is as follows.
In the second section, we will provide the definitions and preliminaries required.
In the third section, we will provide a proof of the theoretical bound on the error of the algorithm with the generated sets and include a demonstration that the bound produces the optimal order of convergence in spaces with polynomial convergence of $\sigma_\bsh$ such as the Korobov space, as defined in Section~\ref{sec:korobovspace}.
In Section~\ref{sec:rational}, we will show that rational generators can also provide the correct bound and, furthermore, a high proportion of them are good when they are selected uniformly from a finite set of possible generators.
Finally, we will look at the Korobov space in more detail and provide parallel results to Theorems~\ref{thm:upperbound}~and~\ref{thm:rationalbound}, which deal with low values of $n$ in contrast to the asymptotic results of Corollaries~\ref{cor:sobolevzetabound} and~\ref{cor:sobolevNzbound}, and establish strong tractability under certain conditions on the weights of the function space.

\section{Background and preliminary results}\label{sec:background}

\subsection{Function space}
For an arbitrary function $f \in L_2([0,1]^d; \CC) = L_2([0, 1]^d)$, we define
\[
    \hat{f}_\bsh := \int_{[0, 1]^d} f(\bsx) \, \rme^{- 2 \pi \rmi \, \bsh \cdot \bsx}  \rd \bsx
    .
\]
Recall that our space and associated norm are defined as
\[
    \Hsigma := \bigg\{ f \in L_2([0, 1]^d) \st \| f \|_\Hsigma^2 := \sum_{\bsh \in \Z^d} |\hat{f}_\bsh|^2 \, \sigma_\bsh^{-2} < \infty \bigg\}
    ,
\]
where $\bssigma \in \ell_2(\Z^d; \R_{> 0})$ is a square summable sequence.

Using Cauchy--Schwarz, we can see that the functions in this space must have an absolutely convergent Fourier series.
If we define a bijection $\bsh_i: \Npos \to \Z^d$, which satisfies $\sigma_{\bsh_i} \ge \sigma_{\bsh_j}$ for $i \le j$, then we can simplify notation so that $\hat{f}_i = \hat{f}_{\bsh_i}$ and $\sigma_i = \sigma_{\bsh_i}$.
We also define
\[
\phii(\bsx)
:=
\rme^{2 \pi \rmi \, \bsh_i \cdot \bsx}
.
\]
This space is a reproducing kernel Hilbert space with finite trace, and the kernel is
\[
    K(\bsx, \bsy) = \sum_{i = 1}^{\infty} \sigma_i^2 \, \phii(\{ \bsx - \bsy \})
    .
\]
Functions in $\Hsigma$ have an absolutely convergent Fourier series and any function with an absolutely convergent Fourier series lies in a space of this form for some $\bssigma \in \ell_2(\Z^d; \R_{> 0})$.

\begin{remark}
    We consider the $d$-dimensional torus, $\bbT^d$, as defined by $[0, 1]^d$ with opposite faces identified.
    All functions in the space $\Hsigma$ have values equal on opposite faces and so are well-defined and continuous on the torus.
    We can just as well say that $\Hsigma \subset C(\bbT^d)$ and so can consider the domain of the functions to be $\bbT^d \simeq [0,1]^d$.
\end{remark}

\subsection{The Korobov space} \label{sec:korobovspace}

A specific example of such a space to which our results apply is called the (\emph{weighted}) \emph{Korobov space}, also called the \emph{unanchored Sobolev space with dominating mixed smoothness on the torus}.
The particular definitions of this space vary in the literature, but we define it as follows.
For positive weights~$\bsgamma = \{ \gamma_\fraku \in (0, 1] \st \fraku \subseteq \{ 1, \ldots, d \} \}$ and smoothness parameter $\alpha > 1/2$ we take the sequence
\begin{align}\label{eq:sigma-Korobov-space}
    \sigma_\bsh^{-1} = \ralpha(\bsh) := \gamma_{\supp(\bsh)}^{-1} \prod_{j \in \supp(\bsh)} |h_j|^\alpha
    ,
\end{align}
where we denote the support of a vector by $\supp(\bsh) := \{ j \in \{ 1, \ldots, d \} : h_j \ne 0 \}$.
Here, $h_j$ represents the $j^\text{th}$ component of the vector $\bsh$.
This should not be confused with $\bsh_j$, with $\bsh$ in bold, which refers to the $j^\text{th}$ element of an ordering on $\Z^d$.
We denote the norm of the Korobov space by $\normKor{\cdot}$.
This is the space of functions from the $d$ dimensional torus $\bbT^d \simeq [0,1]^d$ to $\CC$ which have square-integrable mixed derivatives of order $\alpha$ in each direction.
The weights $\bsgamma$ dictate the relative importance of different subsets of the dimensions and when the weights are all equal to $1$ the space is the unweighted Korobov space.
Here, we have taken the weights to have components less than or equal to $1$ for simplicity; all results can be extended to apply without this restriction by introducing an additional factor depending on $\max_{\fraku \subseteq \{ 1, \ldots, d \}}(\gamma_\fraku)$.

Two features of this function space that enable the application of the forthcoming Corollaries~\ref{cor:sobolevzetabound} and~\ref{cor:sobolevNzbound} are that, in the Korobov space, $\| \bsh_i \|_{\infty} \le i$ for all $i \in \Npos$, and Lemma~\ref{lem:korsigmabound} which bounds the decay of the sequence $\bssigma$ in terms of $i \in \Npos$.
This will use the \emph{weighted hyperbolic cross}, given as
\begin{align} \label{eq:hyperboliccross}
    A_{d,\alpha,\bsgamma}(M)
    :=
    \{ \bsh \in \Z^d \st \ralpha(\bsh) \le M \}
    =
    \{ \bsh_1, \ldots, \bsh_{|A_{d, \alpha, \bsgamma}(M)|} \}
    .
\end{align}

The proof of Lemma~\ref{lem:korsigmabound} requires the following lemma, which appears in other sources (e.g., \cite[Lemma~2.5]{KPV2015}), but we include its proof for completeness.
\begin{lemma}\label{lem:hypcrossbound}
    The unweighted hyperbolic cross, with $M \in \R_{\ge0}$,
    \[
	A_d(M) := \{ \bsh \in \Z^d \st \prod_{\substack{j = 1 \\ h_j \ne 0}}^d |h_j| \le M\}
	,
    \]
    has cardinality $|A_d(M)| = 0$ for $0 \le M < 1$, $|A_d(M)| = 3^d$ for $1 \le M < 2$, and
    \[
	|A_d(M)| \lesssim M \log^{d-1}(M)
	\qquad \text{for } M \ge 2
	,
    \]
    where the implied constant depends on $d$.
\end{lemma}
\begin{proof}
    The cases for $M < 2$ are clear.
    We will prove the case for $M \ge 2$ by induction on~$d$.
    Our induction hypothesis is
    \[
	|A_d(M)| \le \Chyp \, M \log^{d-1}(M)
	\quad
	\text{for all}
	\quad
	M \ge 2
	.
    \]
    Clearly, for $d = 1$ we have $|A_1(M)| = 1 + 2 \floor{M} \le 3 M$, which satisfies the bound.
    For $A_{d+1}(M)$, we have
    \begin{align*}
	A_{d+1}(M)
	=
	&\bigg\{ \bsh \in \Z^{d+1} \st h_{d+1} = 0 \,\text{ and } \prod_{\substack{j = 1 \\ h_j \ne 0}}^d |h_j| \le M \bigg\}
	\\
	&\bigcup\;
	\bigg\{ \bsh \in \Z^{d+1} \st h_{d+1} \ne 0 \,\text{ and } \prod_{\substack{j = 1 \\ h_j \ne 0}}^d |h_j| \le \frac{M}{|h_{d+1}|} \bigg\}
	,
    \end{align*}
    and so, making use of $|A_d(M/h_{d+1})| = 3^d$ for $1 \le M/h_{d+1} < 2$ and the fact that $(M/x) \log^{d-1}(M/x)$ is decreasing on $1 \le x \le M$, we have
    \begin{align*}
	|A_{d+1}(M)|
	&=
	|A_d(M)| + 2\sum_{h_{d+1} = 1}^{\floor{M}} \left| A_d(M/h_{d+1}) \right|
	\\
	&=
	3 \, |A_d(M)| + 2\sum_{h_{d+1} = 2}^{\floor{M / 2}} \left| A_d(M/h_{d+1}) \right| + 3^d \, (\floor{M} - \floor{M/2})
	\\
	&\le
	3 \, \Chyp \, M \log^{d-1}(M) + 2 \, \Chyp \int_1^{M/2} \frac{M}{x} \log^{d-1}(M / x) \rd x + 3^d \, \frac{M+1}{2}
	\\
	&=
	3 \, \Chyp \, M \log^{d-1}(M) + \frac{2 \, \Chyp}{d} \left( M \log^d(M) - M \log^d(2) \right) + 3^d \, \frac{M+1}{2}
	.
    \end{align*}
    This proves the bound.
\end{proof}

We now prove that the sequence $\bssigma$ decays polynomially in the Korobov space.
\begin{lemma}\label{lem:korsigmabound}
    In the Korobov space, i.e., with definition~\eqref{eq:sigma-Korobov-space}, we have
    \[
	\sigma_i \lesssim i^{-\alpha} \log^{\alpha(d-1)}(i)
	,
    \]
    with implied constant depending on $d$, $\alpha$ and $\bsgamma$.
\end{lemma}
\begin{proof}
    Recall the definition of the weighted hyperbolic cross given in~\eqref{eq:hyperboliccross}.
    Since $A_{d,\alpha,\bsgamma}(M) \subseteq A_d((M \max_\fraku \gamma_\fraku)^{1/\alpha})$,
    Lemma~\ref{lem:hypcrossbound} implies
    \begin{equation}\label{eq:HCbound}
	|A_{d,\alpha,\bsgamma}(M)|
	\le
	C_d \,
	(M \max_\fraku \gamma_\fraku)^{1/\alpha} \log^{d-1}((M \max_\fraku \gamma_\fraku)^{1/\alpha})
	\lesssim
	M^{1/\alpha} \log^{d-1}(M^{1/\alpha})
	,
    \end{equation}
    when $(M \max_\fraku \gamma_\fraku)^{1/\alpha} \ge 2$ and the implied constant depends on $d$, $\alpha$ and $\bsgamma$.
    By definition we have $i \le |A_{d,\alpha,\bsgamma}(\ralpha(\bsh_i))|$ and hence for $M = \ralpha(\bsh_i)$ we find from~\eqref{eq:HCbound}
    \begin{align*}
      (\ralpha(\bsh_i) \max_\fraku \gamma_\fraku)^{1/\alpha}
      \ge
      \frac{i}{C_d \log^{d-1}((\ralpha(\bsh_i)\max_\fraku \gamma_\fraku)^{1/\alpha})}
      \ge
      \frac{i}{C_d}
    \end{align*}%
    where the last inequality holds for sufficiently large $i$ such that $(\ralpha(\bsh_i) \max_\fraku \gamma_\fraku)^{1/\alpha} \ge \rme$.
    Hence, from~\eqref{eq:HCbound}, for $i$ sufficiently large we have
    \begin{align*}
      \sigma_{\bsh_i}^{-1} =
      \ralpha(\bsh_i)
      \ge
      \frac{i^\alpha}{C_d^\alpha (\max_\fraku \gamma_\fraku) \log^{\alpha(d-1)}(i/C_d)}
      \gtrsim
      i^\alpha \, \log^{-\alpha(d-1)}(i)
    \end{align*}%
    where the implied constant depends on $d$, $\alpha$ and $\bsgamma$.
\end{proof}

\subsection{Generated sets}

We recall that the generated set for $n \in \Npos$ and $\bszeta \in [0,1]^d$ is defined as the set of points
\[
    \Pnzeta := \left\{ \{ k \bszeta \} \st k = 1, \ldots, n \right\}
    ,
\]
which do not include a point at the origin.
An illustration is given in Figure~\ref{fig:generated-set}.
We will provide a bound on the worst-case error of the least squares algorithm, defined in Section~\ref{sec:leastsquaresalg} below, using these point sets for function approximation in the space $\Hsigma$.

\begin{figure}
    \centering
    \begin{subfigure}{0.33\linewidth}
        \centering
        \includegraphics[width=\linewidth]{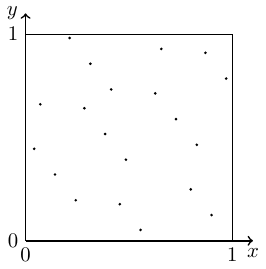}
        \caption{$n = 20$}
    \end{subfigure}%
    \begin{subfigure}{0.33\linewidth}
        \centering
        \includegraphics[width=\linewidth]{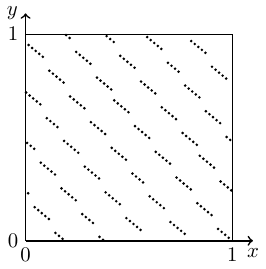}
        \caption{$n = 200$}
    \end{subfigure}%
    \begin{subfigure}{0.33\linewidth}
        \centering
        \includegraphics[width=\linewidth]{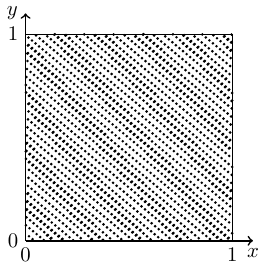}
        \caption{$n = 2000$}
    \end{subfigure}
    \caption{Generated sets with generator $\bszeta = (\sqrt{2} - 1, \sqrt{3} - 1)$ and progressively more sample points.}
    \label{fig:generated-set}
\end{figure}

We will also use the following notation to denote the generated sets with rational components.
For $N \in \Npos$, $\bsz \in \{1,\ldots,N\}^d$ and with $n \le N$, this is given as
\[
    \PnNz = \left\{ \left\{\frac{k \bsz}{N}\right\} \st k = 1, \ldots, n \right\}
    .
\]
This notation is used to highlight the similarity between this set and a rank-$1$ lattice.
The parameter $\bsz$ is referred to as a generating vector and $N$ is the denominator.
This set is a rank-$1$ lattice when $N = n$.

\subsection{The least squares algorithm} \label{sec:leastsquaresalg}

Suppose $P_n = \{ \bsx_1, \ldots, \bsx_n \} \subset [0, 1]^d$ is an arbitrary set of points at which the function will be sampled and recall the definition of the ordering on $\bsh \in \Z^d$, where the values $\bsh_i \in \Z^d$ for $i = 1, \ldots, m$, will be the indices corresponding to the $m$ largest values of $\sigma_\bsh$.
With $m \le n$, we define
\[
    \lsq(f)(\bsx)
    :=
    \sum_{i = 1}^{m} \hat{f}_i^\LS \phii(\bsx)
    ,
\]
where the $\hat{f}_i^\LS$ are chosen to minimise
\begin{equation} \label{eq:LSmin}
    \sum_{k = 1}^{n} \left| \lsq(f)(\bsx_k) - f(\bsx_k) \right|^2
    .
\end{equation}
We define the $n \times m$ matrix
\[
    \Phinm := \big[\phii(\bsx_k) \big]_{k = 1, i = 1}^{n,m}
    ,
\]
and the two vectors
\[
    \hat{\bsf}_m^\LS = \left[ \hat{f}_i^\LS \right]_{i = 1}^m
    \qquad
    \text{and}
    \qquad
    \bsf_\Pn = \left[ f(\bsx_k) \right]_{k = 1}^n
    .
\]
The quantity~\eqref{eq:LSmin} is therefore equal to
\[
    \left\| \Phinm \, \hat{\bsf}_m^\LS - \bsf_\Pn \right\|_{2}^2
    ,
\]
and so if $\hat{\bsf}_m^\LS$ is the least squares solution to this problem, then it satisfies
\begin{equation} \label{eq:normalequation}
    \Phinm^* \Phinm \, \hat{\bsf}_m^\LS = \Phinm^* \, \bsf_\Pn
    .
\end{equation}
Using the Moore--Penrose inverse, we can write
\[
    \hat{\bsf}_m^\LS
    =
    \Phinm^+ \, \bsf_\Pn
    ,
\]
where, if $\Phinm$ has full rank, we have $\Phinm^+ = \left( \Phinm^* \Phinm \right)^{-1} \Phinm^*$.
We extend the above notation to include
\begin{align*}
    \Phininf
    := \big[ \phii(\bsx_k) \big]_{k = 1, i = 1}^{n, \infty}
    \in \CC^{n \times \infty}
    ,
    \quad
    \hat{\bsf}
    :=
    \big[\hat{f}_i\big]_{i = 1}^\infty
    \in \ell_2(\Npos; \CC)
    \quad
    \text{and}
    \quad
    \hat{\bsf}_m
    :=
    \big[\hat{f}_i\big]_{i = 1}^m
    \in \CC^m
    ,
\end{align*}
as well as the truncated vector of functions
\[
    \Phi_m(\bsx) := \left[ \phii(\bsx) \right]_{i = 1}^m
    \quad
    \text{for}
    \quad
    \bsx \in [0, 1]^d
    .
\]
With this in place, we can write
\begin{align} \label{eq:lsmatrix}
    \lsq(f)(\bsx)
    &=
    \sum_{i = 1}^{m} \hat{f}_i^\LS \phii(\bsx)
    \nonumber
    \\
    &=
    \Phi_m^T(\bsx) \, \Phinm^+ \, \bsf_\Pn
    \nonumber
    \\
    &=
    \Phi_m^T(\bsx) \, \Phinm^+ \Phininf \, \hat{\bsf}
    .
\end{align}
We assume that the point sets are chosen such that the matrix $\Phinm$ has full rank and hence $\Phinm^+ = \left( \Phinm^* \Phinm \right)^{-1} \Phinm^*$.
This assumption is reasonable because, as shown in~\cite{K2014}, the set of generators which do not yield a matrix with full rank has Lebesgue measure zero.

\subsection{Error quantity}

In order to determine the accuracy of the algorithm we will use the deterministic worst-case error in $L_2$. This is given as
\[
    \ewor(\lsq) := \sup_{\substack{f \in \Hsigma \\ \| f \|_{\Hsigma} \le 1}} \left\| f - \lsq(f) \right\|_{L_2}
    .
\]
We specify explicitly that this is the deterministic error due to the fact that we demonstrate that good points for the algorithm exist using probabilistic methods but the approximation algorithm itself is deterministic.
For the Korobov space, we write this as
\[
    \eworKor{d}(\lsq) := \sup_{\substack{f \in \KorSpace \\ \normKor{f} \le 1}} \left\| f - \lsq(f) \right\|_{L_2}
    .
\]

\section{The bound for a continuous generator} \label{sec:continuousbound}

Here we provide a proof that there are significantly many generators which provide a certain bound on the worst-case error for generated sets in the space $\Hsigma$.
This bound will be shown to be of the optimal order if the space
has at least polynomial convergence of the sequence~$\bssigma$.
We will use the fact that
\begin{equation} \label{eq:divisorbound}
    \sum_{\substack{i = -n \\ i | n}}^n 1 \le C_\epsilon \, n^\epsilon
    \quad
    \text{for all}
    \quad
    \epsilon > 0
    ,
\end{equation}
where the constant $C_\epsilon \ge 1$ will retain this definition throughout, see~\cite[Section~18.1]{HW1960}.
\begin{theorem}\label{thm:upperbound}
    For the space $\Hsigma$, and $n, m \in \Npos$ with $m \le n$
    and $\bszero \in \{ \bsh_i \st i \le m \}$,
    there are at least a third of the possible choices for $\bszeta \in [0,1]^d$ for which the least squares algorithm $\lsqzeta$, with $\Pnzeta = \{ \{ k \bszeta \} \in [0, 1]^d \st k = 1, \ldots, n \}$, achieves a worst-case error that can be bounded above by
    \begin{align*}
	&\ewor(\lsqzeta)
	\\*
	&\hspace{7mm}\le
	\sigma_{m+1} + \sqrt{\sum_{i > m} \sigma_i^2 + \sqrt{6 C_\epsilon \, n^{1+\epsilon} \sum_{i > m} \sigma_i^4 \, \| \bsh_i \|_{\infty}^\epsilon}} \Bigg/ \sqrt{n - \sqrt{6 C_\epsilon \, n^{1+\epsilon} \, m \max_{i = 1, \ldots, m} \| \bsh_i \|_{\infty}^\epsilon}}
	,
    \end{align*}
    for all $0 < \epsilon \le 1$.

    Furthermore, if
    \begin{equation}\label{eq:hinfbound}
	\| \bsh_i \|_{\infty} \le \Cinfty \, i^r
	\quad
	\text{for some}
	\quad
	\Cinfty \ge 1, \; r > 0
	,
    \end{equation}
    and
    \begin{equation}\label{eq:mbound}
	n^{1-\epsilon} \ge 24 C_\epsilon \Cinfty^\epsilon \, m^{1+r\epsilon}
	,
    \end{equation}
    then the bound will have the form
    \begin{align*}
	\ewor(\lsqzeta)
	&\le
	\sigma_{m+1} + \sqrt{\frac{2}{n}\sum_{i > m} \sigma_i^2 + \sqrt{24 C_\epsilon \Cinfty^\epsilon \, \frac{1}{n^{1-\epsilon}} \sum_{i > m} i^{r\epsilon} \, \sigma_i^4}}
	\\
	&\le
	\sigma_{m+1} + \sqrt{\frac{1}{m}\sum_{i > m} \sigma_i^2 + \sqrt{\frac{1}{m^{1+r\epsilon}} \sum_{i > m} i^{r\epsilon} \, \sigma_i^4}}
	\qquad\qquad
	\text{for all}
	\quad
	0 < \epsilon \le 1
	.
    \end{align*}
\end{theorem}

In order to prove this theorem, we will make use of Lemma~\ref{lem:singvalbounds} which will be given after this proof.
\begin{proof}
    We proceed in a similar manner to the proof of~\cite[Theorem~1]{KU2021a}.
    Set~$\calP_m$ to be the projection
    \[
	\calP_m f = \sum_{i = 1}^m \hat{f}_i \, \phii
    ,
    \]
    then we can bound the error for an arbitrary function $f \in \Hsigma$ as
    \begin{align*}
        \left\| f - \lsqzeta(f) \right\|_{L_2}
        &\le
        \left\| f - \calP_m f  \right\|_{L_2} + \left\| \lsqzeta(f) - \calP_m f  \right\|_{L_2}
	\\
	&=
	\sqrt{\sum_{i > m} |\hat{f}_i|^2} + \left\| \lsqzeta(f) - \calP_m f  \right\|_{L_2}
	\\
	&\le
	\sqrt{\sum_{i > m} \frac{\sigma_{m+1}^2}{\sigma_i^2}  |\hat{f}_i|^2} + \left\| \lsqzeta(f) - \calP_m f  \right\|_{L_2}
	\\
	&\le
	\sigma_{m+1} \big\| f \big\|_{\Hsigma} + \left\| \lsqzeta(f) - \calP_m f  \right\|_{L_2}
	.
    \end{align*}
    Using the fact that $\Imm = \Phinzetam^+ \Phinzetam$, and defining $\Rminf := \big[ \Imm \: \bszero \big] \in \R^{m \times \infty}$, we can rewrite the action of $\calP_m$ as
    \begin{align*}
	(\calP_m f)(\bsx)
	&=
	\Phim^T(\bsx) \, \hat{\bsf}_m
	\\
	&=
	\Phim^T(\bsx) \, \Imm \, \Rminf \, \hat{\bsf}
	\\
	&=
	\Phim^T(\bsx) \, \Phinzetam^+ \Phinzetam \Rminf \, \hat{\bsf}
	\\
	&=
	\Phim^T(\bsx) \, \Phinzetam^+ \Phinzetainf \Rminf^T \Rminf \, \hat{\bsf}
	,
    \end{align*}
    and then, using~\eqref{eq:lsmatrix}, we focus on the second term in the above bound
    \begin{equation*}
	\left\| \lsqzeta(f) - \calP_m f  \right\|_{L_2}
	=
	\left\| \Phim^T(\,\cdot\,) \, \Phinzetam^+ \Phinzetainf (\Iinf - \Rminf^T \Rminf) \, \hat{\bsf} \right\|_{L_2}
	,
    \end{equation*}
    where we use $\Iinf$ to represent the identity, $\Iinf : \ell_2(\Npos; \CC) \to \ell_2(\Npos; \CC)$.
    We note that $\Iinf - \Rminf^T \Rminf$ is a diagonal operator
    \[
    \Iinf - \Rminf^T \Rminf
    =
    \begin{bmatrix}
	\bszero_{m \times m} & \bszero \\
	\bszero & \Iinf
    \end{bmatrix}
    =
    \bszero_{m \times m} \oplus \Iinf
    .
    \]
    We also define the bounded linear operator $\Dlambda := \diag(\sigma_i \st i = 1, \ldots, \infty)$ from $\ell_2(\Npos; \CC)$ to $\ell_2(\Npos; \CC)$ and conclude that
    \begin{align}
    \notag
        \left\| \lsqzeta(f) - \calP_m f  \right\|_{L_2}
	&=
	\left\| \Phinzetam^+ \Phinzetainf (\Iinf - \Rminf^T \Rminf) \hat{\bsf} \right\|_{2}
	\\\notag
	&=
	\left\| \Phinzetam^+ \Phinzetainf (\Iinf - \Rminf^T \Rminf) \Dlambda \, \Dlambda^{-1} \, \hat{\bsf} \right\|_{2}
	\\\notag
	&\le
	\left\| \Phinzetam^+ \right\|_{2 \to 2} \, \left\| \Phinzetainf (\Iinf - \Rminf^T \Rminf) \Dlambda \right\|_{2 \to 2} \, \big\| f \big\|_{\Hsigma}
	\\\label{eq:error-bound-ops}
	&=
	\left\| \Phinzetam^+ \right\|_{2 \to 2} \, \left\| \Phinzetainf \Dlambdam \right\|_{2 \to 2} \, \big\| f \big\|_{\Hsigma}
	,
    \end{align}
    where $\| \cdot \|_{2 \to 2}$ represents the operator norm from Euclidean space to itself, which is notably equal to the largest singular value of the operator, and we defined $\Dlambdam := (\Iinf - \Rminf^T \Rminf) \Dlambda$.
    We, therefore, wish to bound the largest singular values of the operators $\Phinzetam^+$ and $\Phinzetainf \Dlambdam$.

    For the operator $\Phinzetainf \Dlambdam$, we use Chebyshev's inequality and the first part of Lemma~\ref{lem:singvalbounds}, with $a_i = \sigma_i$ and $\calI = \{ i \in \N : i > m \}$, and the fact that the adjoint of an operator shares the same non-zero singular values as the operator itself, to conclude that, for at least $2/3$ of the possible choices of $\bszeta$, we have
    \begin{align*}
	\left\| \Phinzetainf \Dlambdam \right\|_{2 \to 2}^2
	=
	\max_{\substack{\bst \in \CC^n \\ \| \bst \|_{2} = 1}} \left\| \Dlambdam^{\,T} \Phinzetainf^* \, \bst \right\|_{2}^2
	\le
	\sum_{i > m} \sigma_i^2 + \sqrt{6 C_\epsilon \, n^{1+\epsilon} \sum_{i > m} \sigma_i^4 \, \| \bsh_i \|_{\infty}^\epsilon}
	.
    \end{align*}
    Using Chebyshev's inequality and the second part of Lemma~\ref{lem:singvalbounds}, with all $a_i = 1$ and $\calI = \{1,\ldots,m\}$, along with the fact that the largest singular value of the Moore--Penrose inverse of a matrix is the reciprocal of the smallest non-zero singular value of the matrix (recalling that we assume $\Phinm$ is non-singular, since the set of values of $\bszeta \in [0,1]^d$ for which this is not satisfied has measure zero), we can conclude that, for at least $2/3$ of the possible choices of $\bszeta$, we have
    \begin{align*}
        \left\| \Phinzetam^+ \right\|_{2 \to 2}^{-2}
	&=
	\min_{\substack{\bst \in \CC^m \\ \| \bst \|_{2} = 1}} \left\| \Phinzetam \bst \right\|_{2}^2
	\ge
	n - \sqrt{6 C_\epsilon \, n^{1+\epsilon} \, m \max_{i = 1, \ldots, m} \| \bsh_i \|_{\infty}^\epsilon}
	.
    \end{align*}
    The amount of possibilities for which both of these bounds holds is therefore greater than $1/3$, and so, for all such values of $\bszeta$, we have
    \begin{align*}
	&\left\| f - \lsqzeta(f) \right\|_{L_2}
	\\*
	&\hspace{1mm}\le
	\sigma_{m+1} \, \big\| f \big\|_\Hsigma + \left\| \Phinzetainf \Dlambdam \right\|_{2 \to 2} \, \left\| \Phinzetam^+ \right\|_{2 \to 2} \, \big\| f \big\|_\Hsigma
	\\
	&\hspace{1mm}\le
	\Bigg( \sigma_{m+1} + \! \sqrt{\sum_{i > m} \sigma_i^2 + \! \sqrt{6 C_\epsilon \, n^{1+\epsilon} \sum_{i > m} \sigma_i^4 \, \| \bsh_i \|_{\infty}^\epsilon}} \Bigg/ \! \sqrt{n - \! \sqrt{6 C_\epsilon \, n^{1+\epsilon} \, m \max_{i = 1, \ldots, m} \| \bsh_i \|_{\infty}^\epsilon}} \Bigg) \, \big\| f \big\|_\Hsigma
	.
    \end{align*}
    The bound on the error follows from this.
    For the final part of the proof we can recognise that~\eqref{eq:hinfbound} produces
    \begin{align*}
	&\ewor(\lsqzeta)
	\\*
	&\hspace{7mm}\le
	\sigma_{m+1} + \sqrt{\sum_{i > m} \sigma_i^2 + \sqrt{6 C_\epsilon \Cinfty^\epsilon \, n^{1+\epsilon} \sum_{i > m} i^{r \epsilon} \, \sigma_i^4}} \Bigg/ \sqrt{n - \sqrt{6 C_\epsilon \Cinfty^\epsilon \, n^{1+\epsilon} \, m^{1 + r \epsilon} }}
	,
    \end{align*}
    and~\eqref{eq:mbound} ensures that $n - \sqrt{6 C_\epsilon \Cinfty^\epsilon \, n^{1+\epsilon} \, m^{1+r\epsilon}} \ge n/2$.
\end{proof}

In the proof of the previous theorem we needed bounds on two operator norms which hold for a large subset of possible values of $\bszeta \in [0, 1]^d$, see~\eqref{eq:error-bound-ops}.
The following result provides these bounds for a class of operators that includes $\Phinzetam$ and $\Phinzetainf \Dlambdam$.

\begin{lemma} \label{lem:singvalbounds}
    Let $[a_i]_{i \in \calI} \in \R_{\ge 0}^{|\calI|}$, for $\calI \subseteq \N$, be a square summable non-increasing sequence.
    Let $i \mapsto \bsh_i$ be an injection from $\calI$ to $\Z^d$ and set
    \begin{align*}
      A = \big[ a_i \, \rme^{2 \pi \rmi \, k \bsh_i \cdot \bszeta} \big]_{k \in \{ 1, \ldots, n \}, i \in \calI} \; \in \; \CC^{n \times |\calI|}
      .
    \end{align*}
    Assume $\bszeta \in [0, 1]^d$ is a random variable uniformly distributed on $[0, 1]^d$, then the following statements hold.
    \begin{enumerate}
        \item For any $\bst \in \CC^n$,
        \[
	    \bbE_\bszeta\big[ \| A^* \bst \|_2^2 \big]
	    =
	    \sum_{i \in \calI} a_i^2 \, \sum_{k, \ell = 1}^{n} t_k \, \overline{t}_\ell \, \bbone(k \bsh_i = \ell \bsh_i)
	    .
        \]
        Additionally if $\|\bst\|_2 = 1$ then,
        for all $\epsilon > 0$,
        \[
	    \var_\bszeta\big[ \| A^* \bst \|_2^2 \big]
	    \le
	    2 C_\epsilon \, n^{1+\epsilon} \sum_{\substack{i \in \calI \\ \bsh_i \ne \bszero}} a_i^4 \, \| \bsh_i \|_{\infty}^\epsilon
	    .
        \]

        \item For any $\bst \in \CC^{|\calI|}$,
        \[
	    \bbE_\bszeta\big[ \| A \bst \|_{2}^2 \big]
	    =
	    n \sum_{i \in \calI} a_i^2 \, | t_i |^2
	    .
        \]
	Additionally if $\|\bst\|_2 = 1$ then, for all $0 < \epsilon \le 1$,
        \[
        \var_\bszeta\big[ \| A \bst \|_2^2 \big]
        \le
	    2 C_\epsilon \, n^{1+\epsilon} \Big( \max_{i \in \calI} a_i^2 \Big) \, \left( \sum_{i \in \calI} a_i^2 \, \| \bsh_i \|_{\infty}^{2\epsilon} \right)^{1/2} \left( \sum_{i \in \calI} a_i^2 \right)^{1/2}
	    .
        \]
    \end{enumerate}
\end{lemma}

\begin{proof}
    For $\bst \in \CC^n$ we recognise that
    \begin{align*}
	\bbE_\bszeta\left[ \| A^* \bst \|_{2}^2 \right]
	&=
	\sum_{k, \ell = 1}^{n} \sum_{i \in \calI} a_i^2  \, t_k \, \overline{t}_\ell \, \bbE_\bszeta[ \rme^{2 \pi \rmi \, (k - \ell) \bsh_i \cdot \bszeta} ]
	=
	\sum_{k, \ell = 1}^{n} \sum_{i \in \calI} a_i^2 \, t_k \, \overline{t}_\ell \, \bbone(k \bsh_i = \ell \bsh_i)
	.
    \end{align*}
    The variance of $\| A^* \bst \|_{2}^2$ is
    \begin{align*}
	&\var_\bszeta\left[ \| A^* \bst \|_{2}^2 \right]
	\\*
	&\hspace{5mm}=
	\sum_{\substack{k, k' = 1 \\ \ell, \ell' = 1}}^n \sum_{i, j \in \calI} a_i^2 \, a_j^2 \, t_k \, \overline{t}_{k'} \, t_\ell \, \overline{t}_{\ell'} \, \Big( \bbone(\bsh_i (k - k') + \bsh_j (\ell - \ell') = \bszero) - \bbone(k \bsh_i = k' \bsh_i) \, \bbone(\ell \bsh_j = \ell' \bsh_j) \Big)
	\\*
	&\hspace{5mm}=
	\sum_{\substack{k, k' = 1 \\ \ell, \ell' = 1 \\ k \ne k' \\ \ell \ne \ell'}}^n \sum_{\substack{i, j \in \calI \\ \bsh_i \ne \bszero \\ \bsh_j \ne \bszero}} a_i^2 \, a_j^2 \, t_k \, \overline{t}_{k'} \, t_\ell \, \overline{t}_{\ell'} \, \bbone(\bsh_i (k - k') + \bsh_j (\ell - \ell') = \bszero)
	.
    \end{align*}
    We note that
    \begin{align*}
	&\max_{\substack{\bst \in \CC^n \\ \| \bst \|_2 = 1}} \var_\bszeta\left[ \| A^* \bst \|_{2}^2 \right]
	\\*
	&\hspace{5mm}\le
	\max_{\substack{\bst \in \CC^n \\ \| \bst \|_2 = 1}}
	\sum_{\substack{k, k' = 1 \\ \ell, \ell' = 1 \\ k \ne k' \\ \ell \ne \ell'}}^n \sum_{\substack{i, j \in \calI \\ \bsh_i \ne \bszero \\ \bsh_j \ne \bszero}} a_i^2 \, a_j^2  \, |t_k| \, |\overline{t}_{k'}| \, |t_\ell| \, |\overline{t}_{\ell'}| \, \bbone(\bsh_i (k - k') + \bsh_j (\ell - \ell') = \bszero)
	\\
	&\hspace{5mm}=
	\max_{\substack{\bst \in \R_{\ge 0}^n \\ \| \bst \|_2 = 1}}
	\var_\bszeta\left[ \| A^* \, \bst \|_{2}^2 \right]
	.
    \end{align*}
    That is to say that the maximum of the variance for $\bst \in \mathbb{C}^n$ with $\|\bst\|_2 = 1$ is bounded above for some $\bst \in \R_{\ge 0}^n \subset \CC^n$ with $\| \bst \|_2 = 1$ and hence the maximum is actually achieved for some $\bst \in \R_{\ge 0}^n$.
    Therefore, we now assume that $\bst \in \R_{\ge 0}^n$, with $\|\bst\|_2=1$, and bound the variance by the following
    \begin{align*}
	&\var_\bszeta\left[ \| A^* \bst \|_{2}^2 \right]
	\\*
	&\hspace{7mm}=
	\sum_{\substack{k, k' = 1 \\ \ell, \ell' = 1 \\ k \ne k' \\ \ell \ne \ell'}}^n \sum_{\substack{i, j \in \calI \\ \bsh_i \ne \bszero \\ \bsh_j \ne \bszero}} a_i^2 \, a_j^2  \, t_k \, t_{k'} \, t_\ell \, t_{\ell'} \, \bbone(\bsh_i (k - k') + \bsh_j (\ell - \ell') = \bszero)
	\\
	&\hspace{7mm}\le
	2 \sum_{\substack{k, k' = 1 \\ \ell, \ell' = 1 \\ k \ne k' \\ \ell \ne \ell'}}^n \sum_{\substack{i, j \in \calI \\ \bsh_i \ne \bszero \\ \bsh_j \ne \bszero \\ a_i \ge a_j}} a_i^2 \, a_j^2  \, t_k \, t_{k'} \, t_\ell \, t_{\ell'} \, \bbone(\bsh_i (k - k') + \bsh_j (\ell - \ell') = \bszero)
	\\
	&\hspace{7mm}=
	2 \sum_{\substack{k, k' = 1 \\ \ell, \ell' = 1 \\ k \ne k' \\ \ell \ne \ell'}}^n \sum_{\substack{i \in \calI \\ \bsh_i \ne \bszero}} a_i^2 \, t_k \, t_{k'} \, t_\ell \, t_{\ell'} \, \bbone((\ell - \ell') | (k - k') \bsh_i) \underbrace{\sum_{\substack{j \in \calI \\ \bsh_j \ne \bszero \\ a_i \ge a_j}} a_j^2 \, \bbone(\bsh_i (k - k') + \bsh_j (\ell - \ell') = \bszero)}_{\le a_i^2}
	\\
	&\hspace{7mm}\le
	2 \sum_{\substack{k, k' = 1 \\ k \ne k'}}^n \sum_{\substack{i \in \calI \\ \bsh_i \ne \bszero}} a_i^4 \, t_k \, t_{k'} \sum_{\substack{p \in \Z \\ p | (k - k') \bsh_i}} \underbrace{\sum_{\substack{\ell, \ell' = 1 \\ \ell \ne \ell' \\ \ell - \ell' = p}}^n \, t_\ell \, t_{\ell'}}_{\le 1}
	\\
	&\hspace{7mm}\le
	2 C_\epsilon \, n^\epsilon \sum_{\substack{i \in \calI \\ \bsh_i \ne \bszero}} a_i^4 \, \| \bsh_i \|_{\infty}^\epsilon \underbrace{\sum_{\substack{k, k' = 1 \\ k \ne k'}}^n t_k \, t_{k'}}_{\le n}
	\\
	&\hspace{7mm}\le
	2 C_\epsilon \, n^{1+\epsilon} \sum_{\substack{i \in \calI \\ \bsh_i \ne \bszero}}  a_i^4 \, \| \bsh_i \|_{\infty}^\epsilon
	\quad
	\text{for all}
	\quad
	\epsilon > 0
	.
    \end{align*}

    For the second part, we repeat this with $\| A \bst \|_{2}^2$, where now $\bst \in \CC^{|\calI|}$.
    We have
    \begin{align*}
	\bbE_\bszeta\left[ \| A \bst \|_{2}^2 \right]
	&=
	\sum_{i, j \in \calI} \sum_{k = 1}^{n} a_i \, a_j \, t_i \, \overline{t}_j \, \bbE_\bszeta[ \rme^{2 \pi \rmi \, k (\bsh_i - \bsh_j) \cdot \bszeta} ]
	=
	n \sum_{i \in \calI} a_i^2 \, |t_i|^2
	.
    \end{align*}
    Once again, the variance is
    \begin{align*}
	&\var_\bszeta\left[ \| A \bst \|_{2}^2 \right]
	\\*
	&\hspace{5mm}=
	\sum_{\substack{i, j \in \calI \\ i', j' \in \calI}} \; \sum_{k, \ell = 1}^{n} a_i \, a_j \, a_{i'} \, a_{j'}  \, t_i \, \overline{t}_j \, t_{i'} \, \overline{t}_{j'} \, \left[ \bbone\big(k(\bsh_i - \bsh_j) = \ell (\bsh_{i'} - \bsh_{j'}) \big) - \bbone\big( \bsh_i = \bsh_j \big) \, \bbone\big( \bsh_{i'} = \bsh_{j'} \big)\right]
	\\
	&\hspace{5mm}=
	\sum_{\substack{i, j, i', j' \in \calI \\ i \ne j, \; i' \ne j'}} \; \sum_{k, \ell = 1}^{n} a_i \, a_j \, a_{i'} \, a_{j'} \, t_i \, \overline{t}_j \, t_{i'} \, \overline{t}_{j'} \, \bbone\big(k(\bsh_i - \bsh_j) = \ell (\bsh_{i'} - \bsh_{j'}) \big)
	.
    \end{align*}
    Again, we have
    \[
        \max_{\substack{\bst \in \CC^{|\calI|} \\ \| \bst \|_2 = 1}} \var_\bszeta\left[ \| A \bst \|_{2}^2 \right]
	=
        \max_{\substack{\bst \in \R_{\ge 0}^{|\calI|} \\ \| \bst \|_2 = 1}} \var_\bszeta\left[ \| A \bst \|_{2}^2 \right]
	.
    \]
    So we will assume $\bst \in \R_{\ge 0}^{|\calI|}$, and $\|\bst\|_2 = 1$, and bound the variance as follows
    \begin{align*}
	&\var_\bszeta\left[ \| A \bst \|_{2}^2 \right]
	\\*
	&\hspace{10mm}=
	\sum_{\substack{i, j, i', j' \in \calI \\ i \ne j, \; i' \ne j'}} \; \sum_{k, \ell = 1}^{n} a_i \, a_j \, a_{i'} \, a_{j'} \, t_i \, t_j \, t_{i'} \, t_{j'} \, \bbone\big(k(\bsh_i - \bsh_j) = \ell (\bsh_{i'} - \bsh_{j'}) \big)
	\\
	&\hspace{10mm}=
	\sum_{\substack{i, j \in \calI \\ i \ne j}} \sum_{\substack{k, \ell = 1 \\ \ell | k (\bsh_i - \bsh_j)}}^{n} a_i \, a_j \, t_i \, t_j \hspace{-4mm} \sum_{\substack{i', j' \in \calI \\ i' \ne j' \\ k(\bsh_i - \bsh_j) = \ell (\bsh_{i'} - \bsh_{j'})}} \hspace{-4mm} a_{i'} \, a_{j'} \, t_{i'} \, t_{j'}
	\\
	&\hspace{10mm}\le
	\sum_{\substack{i, j \in \calI \\ i \ne j}} \sum_{\substack{k, \ell = 1 \\ \ell | k (\bsh_i - \bsh_j)}}^{n} a_i \, a_j \, t_i \, t_j
	\vast( \sum_{\substack{i', j' \in \calI \\ i' \ne j' \\ k(\bsh_i - \bsh_j) = \ell (\bsh_{i'} - \bsh_{j'})}} \hspace{-4mm} a_{i'}^2 \, t_{i'}^2 \vast)^{1/2}
	\vast( \sum_{\substack{i', j' \in \calI \\ i' \ne j' \\ k(\bsh_i - \bsh_j) = \ell (\bsh_{i'} - \bsh_{j'})}} \hspace{-4mm} a_{j'}^2 \, t_{j'}^2 \vast)^{1/2}
	\\
	&\hspace{10mm}\le
	\sum_{\substack{i, j \in \calI \\ i \ne j}} \sum_{\substack{k, \ell = 1 \\ \ell | k (\bsh_i - \bsh_j)}}^{n} a_i \, a_j \, t_i \, t_j \, \max_{i' \in \calI} a_{i'}^2
	\\
	&\hspace{10mm}\le
	C_\epsilon \, n^{1+\epsilon} \Big( \max_{i' \in \calI} a_{i'}^2 \Big) \sum_{\substack{i, j \in \calI \\ i \ne j}} a_i \, a_j \, t_i \, t_j \, \| \bsh_i - \bsh_j \|_{\infty}^\epsilon
	\\
	&\hspace{10mm}\le
	C_\epsilon \, n^{1+\epsilon} \Big( \max_{i' \in \calI} a_{i'}^2 \Big) \sum_{i, j \in \calI} a_i \, a_j \, t_i \, t_j \, \big( \| \bsh_i \|_{\infty} + \| \bsh_j \|_{\infty} \big)^\epsilon
	\\
	&\hspace{10mm}\le
	C_\epsilon \, n^{1+\epsilon} \Big( \max_{i' \in \calI} a_{i'}^2 \Big) \sum_{i, j \in \calI} a_i \, a_j \, t_i \, t_j \, \big( \| \bsh_i \|_{\infty}^\epsilon + \| \bsh_j \|_{\infty}^\epsilon \big)
	\\
	&\hspace{10mm}\le
	2 C_\epsilon \, n^{1+\epsilon} \Big( \max_{i' \in \calI} a_{i'}^2 \Big) \Bigg( \sum_{i \in \calI} \| \bsh_i \|_{\infty}^\epsilon \, a_i \, t_i \Bigg) \Bigg( \sum_{j \in \calI} a_j \, t_j \Bigg)
	\\
	&\hspace{10mm}\le
	2 C_\epsilon \, n^{1+\epsilon} \Big( \max_{i' \in \calI} a_{i'}^2 \Big) \Bigg( \sum_{i \in \calI} \| \bsh_i \|_{\infty}^{2 \epsilon} \, a_i^2 \Bigg)^{1/2} \Bigg( \sum_{j \in \calI} a_j^2 \Bigg)^{1/2}
	\quad
	\text{for all}
	\quad
	0 < \epsilon \le 1
	.
    \end{align*}
    This completes the proof.
\end{proof}

We can now apply Theorem~\ref{thm:upperbound} to spaces with polynomial convergence of the sequence~$\bssigma$.
\begin{corollary} \label{cor:sobolevzetabound}
    If the sequence $\bssigma$ is such that the following conditions hold
    \begin{enumerate}
	\item $\| \bsh_i \|_{\infty} \le \Cinfty \, i^r$ for some $\Cinfty \ge 1$, $r > 0$, and
	\item $\sigma_i \lesssim i^{-\alpha} \log^{\beta}(i)$,
    \end{enumerate}
    where $\alpha > 1/2$ and $\beta \in \R$,
    then for sufficiently large $n \in \Npos$, there exists $m \le n$ which satisfies~\eqref{eq:mbound}
    such that the bound from Theorem~\ref{thm:upperbound} becomes
    \[
	\ewor(\lsqzeta)
	\lesssim
	n^{-\frac{1-\epsilon}{1+r\epsilon} \alpha} \log^{\beta}(n)
	\quad
	\text{for all}
	\quad
	0 < \epsilon \le 1
	,
    \]
    where the implied constant goes to infinity as $\epsilon$ goes to $0$.
\end{corollary}
\begin{proof}
    Here we will use the fact that
    \[
	\sum_{i > m} i^{-a} \log^b(i) \lesssim m^{1-a} \log^b(m)
	,
    \]
    when $a > 1$ and $b \in \R$.
    We choose $m = \lfloor n^{(1-\epsilon)/(1+r\epsilon)} / (24 C_\epsilon \Cinfty^\epsilon)^{1/(1+r\epsilon)} \rfloor$ and $n$ large enough so that $\bszero \in \{ \bsh_i \st i \le m \}$.
    We have
    \begin{align*}
	\frac{1}{m} \sum_{i > m} \sigma_i^2
	&\lesssim
	m^{-1} \, m^{1-2\alpha} \log^{2 \beta}(m)
	\lesssim
	n^{-2\frac{1-\epsilon}{1+r\epsilon} \alpha} \log^{2\beta}(n)
    \end{align*}
    and
    \begin{align*}
	\frac{1}{m^{1+r\epsilon}} \sum_{i > m} i^{r\epsilon - 4\alpha} \log^{4\beta}(i)
	&\lesssim
	m^{-1-r\epsilon} \, m^{r\epsilon + 1 - 4\alpha} \log^{4\beta}(m)
	\lesssim
	n^{-4\frac{1-\epsilon}{1+r\epsilon} \alpha} \log^{4\beta}(n)
	.
    \end{align*}
    This concludes the proof.
\end{proof}

The corollary applies to the Korobov space, as defined in Section~\ref{sec:korobovspace}, since it satisfies $\| \bsh_i \|_{\infty} \le i$ and $\sigma_i \lesssim i^{-\alpha} \log^{\alpha(d-1)}(i)$.
This will hold with $\Cinfty = 1$, $r = 1$ and $\beta = \alpha(d-1)$.

We note here that this corollary (and Corollary~\ref{cor:sobolevNzbound} below) are asymptotic bounds and do not necessarily apply for lower values of $n$.
They also do not deal with the issue of tractability as the bounds may increase exponentially in $d$.
These issues are addressed in Section~\ref{sec:korobovcorollaries} for the weighted Korobov space.

\section{Generated sets with rational generators} \label{sec:rational}

We will now show, that the points used in Theorem~\ref{thm:upperbound} can have a rational generator and furthermore we will show that there are significantly many generators with a denominator which is bounded by some function depending on $n$.
For some prime $N \in \Npos$, we can deviate from the previous analysis by choosing a generator uniformly from the subset $\{ i/N \st i = 1, \ldots, N \}^d \subset [0, 1]^d$, rather than from the entire unit cube $[0, 1]^d$.
Recall that the corresponding generated set is given as
\[
    \PnNz = \left\{ \left\{\frac{k \bsz}{N}\right\} \st k = 1, \ldots, n \right\}
    ,
\]
where $\bsz \in \{ 1, \ldots, N \}^d$ and we have in mind that $N \ge n$.
For the following, the notation $\bsh \equiv_N \bsh'$ is taken to mean that all components of the vector $\bsh$ are equivalent to those of $\bsh'$ modulo $N$, and similarly, $\bsh \not\equiv_N \bsh'$ means that at least one component is not equivalent modulo~$N$.

\begin{theorem}\label{thm:rationalbound}
    For the space $\Hsigma$ for which \eqref{eq:hinfbound} holds, i.e.,
    \begin{equation*}
	\| \bsh_i \|_{\infty} \le \Cinfty \, i^r
	\quad
	\text{for some}
	\quad
	\Cinfty \ge 1, \; r > 0
	,
    \end{equation*}
    and with $n, m \in \Npos$, $n \ge m$, satisfying~\eqref{eq:mbound}, i.e.,
    \begin{equation*}
	n^{1-\epsilon} \ge 24 C_\epsilon \Cinfty^\epsilon \, m^{1+r\epsilon}
	,
    \end{equation*}
    and $\bszero \in \{ \bsh_i \st i \le m \}$,
    and a prime $N$, which satisfies
    \begin{equation} \label{eq:firstNbound}
	N > 4 \, n \, \| \bsh_i \|_\infty
	\quad
	\text{for all}
	\quad
	i \le m
	,
    \end{equation}
    there are at least a third of the possible choices for $\bsz \in \{ 1, \ldots, N \}^d$ for which the least squares algorithm $\lsqNz$ achieves a worst-case error that can be bounded above by
    \begin{align*}
	&\ewor(\lsqNz)
	\\
	&\hspace{7mm}\le
	\sigma_{m+1} + \sqrt{2 \sum_{\substack{i > m \\ \bsh_i \equiv_N \bszero}} \sigma_i^2 + \frac{2}{n} \sum_{i > m} \sigma_i^2 + \sqrt{\frac{24 C_\epsilon \Cinfty^\epsilon}{n^{1-\epsilon}} \sum_{i > m} \sigma_i^4 \, i^{r \epsilon} + \frac{24}{n} \sum_{\substack{i, j > m \\ 2 n \| \bsh_j \|_{\infty} > N}} \sigma_i^2 \, \sigma_j^2}}
	,
    \end{align*}
    for all $0 < \epsilon \le 1$.
\end{theorem}

\begin{proof}
    The proof of this theorem is near identical to the second part of Theorem~\ref{thm:upperbound} with the application of Lemma~\ref{lem:rationalbound}, given below, instead of Lemma~\ref{lem:singvalbounds}.
    As in the proof of Theorem~\ref{thm:upperbound}, we have
    \begin{align*}
	\ewor(\lsqNz)
	\le
	\sigma_{m+1} + \left\| \PhinNzinf \, \Dlambdam \right\|_{2 \to 2} \, \left\| \PhinNzm^+ \right\|_{2 \to 2}
	.
    \end{align*}
    Using Chebyshev's inequality, the second part of Lemma~\ref{lem:rationalbound}, together with~\eqref{eq:hinfbound} and~\eqref{eq:mbound}, implies
    that at least two thirds of the possible choices for $\bsz$ satisfy
    $$\| \PhinNzm^+ \|_{2 \to 2}^2 \le 2/n .$$
    Furthermore, the first part of Lemma~\ref{lem:rationalbound}, together with~\eqref{eq:hinfbound}, implies that at least two thirds of the possible choice for $\bsz$ satisfy
    \begin{align*}
	\left\| \PhinNzinf \, \Dlambdam \right\|_{2 \to 2}^2
	&\le
	n \sum_{\substack{i > m \\ \bsh_i \equiv_N \bszero}} \sigma_i^2 + \sum_{i > m} \sigma_i^2 + \sqrt{\frac{6 C_\epsilon \Cinfty^\epsilon}{n^{-1-\epsilon}} \sum_{i > m} i^{r \epsilon} \, \sigma_i^4  + 6n \sum_{\substack{i, j > m \\ 2n \| \bsh_j \|_{\infty} > N}} \sigma_i^2 \, \sigma_j^2}
	.
    \end{align*}
    Hence the result follows.
\end{proof}

The following lemma which was used in the previous proof is analogous to Lemma~\ref{lem:singvalbounds}.

\begin{lemma} \label{lem:rationalbound}
    Let $[a_i]_{i \in \calI} \in \R_{\ge 0}^{|\calI|}$, for $\calI \subseteq \N$, be a square summable non-increasing sequence.
    Let $i \mapsto \bsh_i$ be a injection from $\calI$ to $\Z^d$, and set
    \begin{align*}
      A = \big[ a_i \, \rme^{2 \pi \rmi \, k \bsh_i \cdot \bszeta} \big]_{k \in \{ 1, \ldots, n \}, i \in \calI} \; \in \; \CC^{n \times |\calI|}
      .
    \end{align*}
    Assume $\bsz$ is chosen uniformly from $\{ 1, \ldots, N \}^d$ where $N > 2n$ is prime,
    then the following statements hold.
    \begin{enumerate}
        \item For any $\bst \in \CC^n$, satisfying $\| \bst \|_{2} = 1$,
        \[
            \bbE_\bsz\big[ \| A^* \bst \|_2^2 \big]
            =
            \sum_{k, \ell = 1}^{n} \sum_{\substack{i \in \calI \\ \bsh_i \equiv_N \bszero}} a_i^2 \, t_k \, \overline{t}_\ell + \sum_{\substack{i \in \calI \\ \bsh_i \not\equiv_N \bszero}} a_i^2
            \le
            2 C_\epsilon n^{1+\epsilon} \sum_{\substack{i \in \calI \\ \bsh_i \not\equiv_N \bszero}} a_i^4 \, \| \bsh_i \|_{\infty}^\epsilon + 2 n \sum_{\substack{i \in \calI \\ j \in \calI \setminus \calI' \\ \bsh_i \not\equiv_N \bszero \\ \bsh_j \not\equiv_N \bszero}} a_i^2 \, a_j^2
            .
        \]
        Let $\calI' = \{ i \in \calI : 2 n \| \bsh_i \|_{\infty} < N \}$, then, for all $\epsilon > 0$,
        \[
            \var_\bsz\big[\| A^* \bst \|_2^2\big]
            \le
            2 C_\epsilon n^{1+\epsilon} \sum_{\substack{i \in \calI \\ \bsh_i \not\equiv_N \bszero}} a_i^4 \, \| \bsh_i \|_{\infty}^\epsilon + 2 n \sum_{\substack{i \in \calI \\ j \in \calI \setminus \calI' \\ \bsh_i \not\equiv_N \bszero \\ \bsh_j \not\equiv_N \bszero}} a_i^2 \, a_j^2
            .
        \]
        \item If $N$ satisfies $N > 4 n \max_{i \in \calI} \| \bsh_i \|_\infty$, then for any $\bst \in \CC^{|\calI|}$,
        \[
	    \bbE_\bsz\big[ \| A \bst \|_2^2 \big]
	    =
	    n \sum_{i \in \calI} a_i^2 \left| t_i \right|^2
	    .
        \]
	    Additionally if $\|\bst\|_2 = 1$ then, for all $0 < \epsilon \le 1$,
        \[
        \var_\bsz\big[ \| A \bst \|_2^2 \big]
        \le
	    2 C_\epsilon n^{1+\epsilon} \Big( \max_{i \in \calI} a_i^2 \Big) \left( \sum_{i \in \calI} a_i^2 \, \| \bsh_i \|_{\infty}^{2\epsilon} \right)^{1/2} \left( \sum_{i \in \calI} a_i^2 \right)^{1/2}
	    .
        \]
    \end{enumerate}
\end{lemma}
\begin{proof}
    For the first part, for any $\bst \in \CC^n$ with $\| \bst \|_2 = 1$, we recognise that
    \begin{align*}
        \bbE_\bsz[ \| A^* \bst \|_{2}^2 ]
	&=
	\frac{1}{N^d} \sum_{\bsz \in \{ 1, \ldots, N \}^d} \sum_{k, \ell = 1}^{n} \sum_{i \in \calI} a_i^2 \, t_k \, \overline{t}_\ell \, \rme^{2 \pi \rmi \, (k - \ell) \bsz \cdot \bsh_i / N}
	\\
	&=
	\sum_{k, \ell = 1}^{n} \sum_{i \in \calI} a_i^2 \, t_k \, \overline{t}_\ell \, \bbone((k - \ell) \bsh_i \equiv_N \bszero)
	\\
	&=
	\sum_{k, \ell = 1}^{n} \sum_{\substack{i \in \calI \\ \bsh_i \equiv_N \bszero}} a_i^2 \, t_k \, \overline{t}_\ell + \sum_{\substack{i \in \calI \\ \bsh_i \not\equiv_N \bszero}} a_i^2
	\\
	&\le
	n \sum_{\substack{i \in \calI \\ \bsh_i \equiv_N \bszero}} a_i^2 + \sum_{\substack{i \in \calI \\ \bsh_i \not\equiv_N \bszero}} a_i^2
	.
    \end{align*}
    As with Lemma~\ref{lem:singvalbounds}, we can calculate the variance
    \begin{align*}
	&\var_\bsz[ \| A^* \bst \|_{2}^2 ]
	\\*
	&\hspace{4mm}=
	\sum_{\substack{k, k' = 1 \\ \ell, \ell' = 1}}^n \sum_{i, j \in \calI} a_i^2 \, a_j^2 \, t_k \, \overline{t}_{k'} \, t_\ell \, \overline{t}_{\ell'} \, \Big( \bbone(\bsh_i (k - k') \equiv_N \bsh_j (\ell - \ell')) - \bbone(k \bsh_i \equiv_N k' \bsh_i) \, \bbone(\ell \bsh_j \equiv_N \ell' \bsh_j) \Big)
	\\*
	&\hspace{4mm}=
	\sum_{\substack{k, k' = 1 \\ \ell, \ell' = 1 \\ k \ne k' \\ \ell \ne \ell'}}^n \sum_{\substack{i, j \in \calI \\ \bsh_i \not\equiv_N \bszero \\ \bsh_j \not\equiv_N \bszero}} a_i^2 \, a_j^2  \, t_k \, \overline{t}_{k'} \, t_\ell \, \overline{t}_{\ell'} \, \bbone(\bsh_i (k - k') \equiv_N \bsh_j (\ell - \ell'))
	.
    \end{align*}
    Similar to the proof of Lemma~\ref{lem:singvalbounds}, we note that
    \[
	\max_{\substack{\bst \in \CC^n \\ \| \bst \|_2 = 1}} \var_\bsz[ \| A^* \bst \|_{2}^2 ]
	=
	\max_{\substack{\bst \in \R^n_{\ge 0} \\ \| \bst \|_2 = 1}} \var_\bsz[ \| A^* \bst \|_{2}^2 ]
	.
    \]
    So we assume $\bst \in \R^n_{\ge 0}$, with $\|\bst\|_2 = 1$, and continue to bound the variance, using $\calI' \subseteq \calI$,
    \begin{align*}
	&\var_\bsz[ \| A^* \bst \|_{2}^2 ]
	\\*
	&\hspace{7mm}\le
	\sum_{\substack{k, k' = 1 \\ \ell, \ell' = 1 \\ k \ne k' \\ \ell \ne \ell'}}^n \sum_{\substack{i, j \in \calI' \\ \bsh_i \not\equiv_N \bszero \\ \bsh_j \not\equiv_N \bszero}} a_i^2 \, a_j^2 \, t_k \, t_{k'} \, t_\ell \, t_{\ell'} \, \bbone(\bsh_i (k - k') \equiv_N \bsh_j (\ell - \ell'))
	\\*
	&\hspace{16mm}+
	2\sum_{\substack{k, k' = 1 \\ \ell, \ell' = 1 \\ k \ne k' \\ \ell \ne \ell'}}^n \sum_{\substack{i \in \calI \\ j \in \calI \setminus \calI' \\ \bsh_i \not\equiv_N \bszero \\ \bsh_j \not\equiv_N \bszero}} a_i^2 \, a_j^2 \, t_k \, t_{k'} \, t_\ell \, t_{\ell'} \, \bbone(\bsh_i (k - k') \equiv_N \bsh_j (\ell - \ell'))
	.
    \end{align*}
    To bound the first part, we recognise that this is a smaller quantity than the variance from the first part of Lemma~\ref{lem:singvalbounds}.
    We have
    \begin{align*}
	&\sum_{\substack{k, k' = 1 \\ \ell, \ell' = 1 \\ k \ne k' \\ \ell \ne \ell'}}^n \sum_{\substack{i, j \in \calI' \\ \bsh_i \not\equiv_N \bszero \\ \bsh_j \not\equiv_N \bszero}} a_i^2 \, a_j^2 \, t_k \, t_{k'} \, t_\ell \, t_{\ell'} \, \bbone(\bsh_i (k - k') \equiv_N \bsh_j (\ell - \ell'))
	\\
	&\hspace{10mm}=
	\sum_{\substack{k, k' = 1 \\ \ell, \ell' = 1 \\ k \ne k' \\ \ell \ne \ell'}}^n \sum_{\substack{i, j \in \calI' \\ \bsh_i \not\equiv_N \bszero \\ \bsh_j \not\equiv_N \bszero}} a_i^2 \, a_j^2 \, t_k \, t_{k'} \, t_\ell \, t_{\ell'} \, \bbone(\bsh_i (k - k') = \bsh_j (\ell - \ell'))
	\\
	&\hspace{10mm}\le
	2 C_\epsilon n^{1+\epsilon} \sum_{\substack{i \in \calI \\ \bsh_i \ne \bszero}}  a_i^4 \, \| \bsh_i \|_{\infty}^\epsilon
	\quad
	\text{for all}
	\quad
	\epsilon > 0
	,
    \end{align*}
    where, for the equality, we used the fact that, for $i, j \in \calI'$, then
    \[
    \bbone(\bsh_i (k - k') \equiv_N \bsh_j (\ell - \ell'))
    =
    \bbone(\bsh_i (k - k') = \bsh_j (\ell - \ell'))
    .
    \]
    For the second term we proceed similarly
    \begin{align*}
	&2\sum_{\substack{k, k' = 1 \\ \ell, \ell' = 1 \\ k \ne k' \\ \ell \ne \ell'}}^n \sum_{\substack{i \in \calI \\ j \in \calI \setminus \calI' \\ \bsh_i \not\equiv_N \bszero \\ \bsh_j \not\equiv_N \bszero}} a_i^2 \, a_j^2 \, t_k \, t_{k'} \, t_\ell \, t_{\ell'} \, \bbone(\bsh_i (k - k') \equiv_N \bsh_j (\ell - \ell'))
	\\
	&\hspace{10mm}=
	2 \sum_{\substack{i \in \calI \\ j \in \calI \setminus \calI' \\ \bsh_i \not\equiv_N \bszero \\ \bsh_j \not\equiv_N \bszero}} a_i^2 \, a_j^2 \sum_{\substack{k, k' = 1 \\ k \ne k'}}^n t_k \, t_{k'} \underbrace{\sum_{\substack{\ell, \ell' = 1 \\ \ell \ne \ell' \\ \bsh_i (k - k') \equiv_N \bsh_j (\ell - \ell')}}^n \hspace{-4mm} t_\ell \, t_{\ell'}}_{\le 1}
	\\
	&\hspace{10mm}\le
	2 n \sum_{\substack{i \in \calI \\ j \in \calI \setminus \calI' \\ \bsh_i \not\equiv_N \bszero \\ \bsh_j \not\equiv_N \bszero}} a_i^2 \, a_j^2
	.
    \end{align*}

    The second conclusion of this lemma is identical to the second part of Lemma~\ref{lem:singvalbounds}.
    In order for the proof of Lemma~\ref{lem:singvalbounds} to hold, we require that the distribution of $\bszeta$ satisfies
    \begin{equation} \label{eq:randcond}
        \bbE_\bszeta[ \rme^{2 \pi \rmi \, \bsh \cdot \bszeta} ] = 0
	\quad \text{when} \quad
	 \bsh \ne \bszero
	,
    \end{equation}
    where $\bsh \in \Z^d$ is a vector with components smaller than $4 n \| \bsh_i \|_\infty$ for all $i \in \calI$, either being $\bsh = k (\bsh_i - \bsh_j)$ or $\bsh = k (\bsh_i - \bsh_j) - \ell (\bsh_{i'} - \bsh_{j'})$ for $k,  \ell \in \{1,\ldots,n\}$ and $i,i',j,j' \in \calI$ in the proof of Lemma~\ref{lem:singvalbounds}.
    For $N > 4 n \max_{i \in \calI} \| \bsh_i \|_\infty$, the expectations~\eqref{eq:randcond} hold with $\bszeta = \bsz / N$ and $\bsz$ chosen uniformly from $\{ 1, \ldots, N \}^d$, therefore the proof from Lemma~\ref{lem:singvalbounds} applies here.
\end{proof}

We now apply Theorem~\ref{thm:rationalbound} to function spaces $\Hsigma$ with $\bssigma$ decaying polynomially, like the Korobov space, to obtain a similar result as in Corollary~\ref{cor:sobolevzetabound}.

\begin{corollary} \label{cor:sobolevNzbound}
    If the sequence $\bssigma$ is such that the following conditions hold
    \begin{enumerate}
	\item $\| \bsh_i \|_{\infty} \le \Cinfty \, i^r$ for some $\Cinfty \ge 1$, $r > 0$, and
	\item $\sigma_i \lesssim i^{-\alpha} \log^{\beta}(i)$,
    \end{enumerate}
    where $\alpha > 1/2$ and $\beta \in \R$,
    then for sufficiently large $n \in \Npos$, there exists $m \le n$ which satisfies~\eqref{eq:mbound},
    and some prime $N > 4 n m^{2\alpha r/(2 \alpha-1)}$, such that
    the bound from Theorem~\ref{thm:rationalbound} becomes
    \[
	\ewor(\lsqNz)
	\lesssim
	n^{-\frac{1-\epsilon}{1+r\epsilon} \alpha} \log^\beta(n)
	\quad
	\text{for all}
	\quad
	0 < \epsilon \le 1
	,
    \]
    where the implied constant goes to infinity as $\epsilon$ goes to $0$.
\end{corollary}
\begin{proof}
    We choose $m = \lfloor n^{(1-\epsilon)/(1+r\epsilon)} / (24 C_\epsilon \Cinfty^\epsilon)^{1/(1+r\epsilon)} \rfloor$ and $n$ large enough so that ${\bszero \in \{ \bsh_i \st i \le m \}}$ and $m^r > \Cinfty^{2 \alpha - 1}$.
    We note that the condition $N > 4n m^{2\alpha r/(2\alpha - 1)}$, together with $m^r > \Cinfty^{2\alpha - 1}$ and the first assumption on $\bssigma$, implies~\eqref{eq:firstNbound}.
    In addition to the bounds in the proof of Corollary~\ref{cor:sobolevzetabound}, we have
    \begin{align*}
        \sum_{\substack{i > m \\ \bsh_i \equiv_N \bszero}} \sigma_i^2
	\le
        \sum_{\substack{i > m \\ \| \bsh_i \|_{\infty} \ge N}} \sigma_i^2
	\le
        \sum_{\substack{i > m \\ i \ge (N/\Cinfty)^{1/r}}} \sigma_i^2
	&\lesssim
	N^{(1-2\alpha)/r} \log^{2\beta}(N)
	\\
	&\lesssim
        m^{-2\alpha} n^{(1-2\alpha)/r} \log^{2\beta}(n)
	\lesssim
        n^{-2 \frac{1-\epsilon}{1+r\epsilon} \alpha} \log^{2\beta}(n)
	.
    \end{align*}
    Also, we have
    \begin{align*}
        \frac{1}{n} \sum_{\substack{i, j > m \\ 2 n \| \bsh_j \|_{\infty} > N}} \sigma_i^2 \, \sigma_j^2
        &\le
	n^{-1} \left( \sum_{i > m} \sigma_i^2 \right) \left( \sum_{j > (N/(2 \Cinfty n))^{1/r}} \sigma_j^2 \right)
        \\
        &\lesssim
	n^{-1} m^{1-2\alpha} \log^{2\beta}(m) \, (N / (2 n))^{(1 - 2\alpha)/r} \log^{2\beta}(N / (2 n))
        \\
        &\lesssim
        n^{-4 \frac{1-\epsilon}{1+r\epsilon} \alpha} \log^{4\beta}(n)
        .
    \end{align*}
    This concludes the proof.
\end{proof}

\section{Bounds specific to the weighted Korobov space}\label{sec:korobovcorollaries}

Here, we provide parallel results to the above two corollaries for the weighted Korobov space.
Where the previous Corollaries~\ref{cor:sobolevzetabound} and~\ref{cor:sobolevNzbound} also apply to the Korobov space, here we explicitly make use of the function space weights.
The results here also provide more restrictive constants which apply for all $n$, so long as $\bszero \in \{ \bsh_i \st i \le m \}$, and demonstrate strong tractability dependent on the decay of the function space weights.
We will use an alternative bound to~\eqref{eq:HCbound} on the size of the weighted hyperbolic cross~\eqref{eq:hyperboliccross}, given by~\cite[Lemma~5.2]{KMN2024},
\begin{equation} \label{eq:HCbound2}
    |A_{d, \alpha, \bsgamma}(M)|
    \le
    M^{1/\lambda}
    \sum_{\bsh \in \Z^d} \ralpha^{-1/\lambda}(\bsh)
    \quad
    \text{for all}
    \quad
    \lambda < \alpha
    .
\end{equation}
In the Korobov space, condition~\eqref{eq:hinfbound} holds with $\Cinfty = 1$ and $r = 1$, and so condition~\eqref{eq:mbound} becomes
\begin{equation} \label{eq:kormbound}
    n^{1-\epsilon} \ge 24 C_\epsilon \, m^{1+\epsilon}
    .
\end{equation}
For the following, we denote
\begin{equation}\label{eq:mu_quantity}
    \muquant(\lambda)
    :=
    \sum_{\bsh \in \Z^d}  \ralpha^{-1/\lambda}(\bsh)
    =
    \sum_{\fraku \subseteq \{1,\ldots,d\}} \gamma_{\fraku}^{1/\lambda} \big(2\zeta(\alpha/\lambda) \big)^{|\fraku|}
    ,
\end{equation}
where $\zeta(\cdot)$ is the Riemann zeta function.
\begin{corollary} \label{cor:korobovbound}
    Let $d \in \Npos$, $\alpha > 1/2$, and $\bsgamma = \{ \gamma_\fraku \}_{\fraku \subseteq \{ 1, \ldots, d \}}$ be a set of weights with $\gamma_\fraku \in (0, 1]$.
    For the Korobov space, $\KorSpace$,
    if $n \in \Npos$ is large enough such that
    \begin{align} \label{eq:nMcond}
	\left| A_{d, \alpha, \bsgamma}\left( M \right) \right|^\lambda \muquant^{-\lambda}(\lambda) > \gamma_\emptyset^{-1}
	\quad
	\text{where}
	\quad
	M = \left( \frac{n^{1-\epsilon}}{24 C_\epsilon} \right)^{\lambda/(1+\epsilon)} \muquant^{-\lambda}(\lambda)
	,
    \end{align}
    then, for $m = |A_{d, \alpha, \bsgamma}(M)|$, the bound from Theorem~\ref{thm:upperbound} has the following form
    \[
	\eworKor{d}(\lsqzeta)
	\le
	3 \left( 24 C_\epsilon \, \muquant(\lambda) \, n^{-\frac{1-\epsilon}{1+\epsilon}} \right)^\lambda
	,
    \]
    for $0 < \epsilon \le 1$ and $1/2 < \lambda < \alpha$, and $\muquant$ defined by~\eqref{eq:mu_quantity}.
    Therefore, at least a third of the possible choices of $\bszeta \in [0, 1]^d$ will produce this bound.
    If we extend $\{ \gamma_\fraku \}_{\fraku \subseteq \Npos}$ to a system of weights for all finite subsets of $\Npos$, and if
    \[
       \sum_{\substack{\fraku \subset \Npos \\ |\fraku| < \infty}} \gamma_\fraku^{1/\alpha} \big( 2 \zeta(\alpha/\lambda) \big)^{|\fraku|}
       <
       \infty
    \]
    holds, then the bound can be stated with an absolute constant independent of the dimension.
\end{corollary}
\begin{proof}
    We slightly edit the bound from Theorem~\ref{thm:upperbound} on $\| f - \calP_m f \|_{L_2}^2$ as follows:
    \begin{align*}
        \| f - \calP_m f \|_{L_2}^2
	&=
	\sum_{\substack{\bsh \in \Z^d \\ \ralpha(\bsh) > M}} |\hat{f}_\bsh|^2
	\le
	\sum_{\substack{\bsh \in \Z^d \\ \ralpha(\bsh) > M}} \left(\frac{\ralpha(\bsh)}{M}\right)^2 |\hat{f}_\bsh|^2
	\le
	\frac{1}{M^2} \big\| f \big\|_{d, \alpha, \bsgamma}^2
	.
    \end{align*}
    Here, we use~\eqref{eq:HCbound2} to see that~\eqref{eq:kormbound} is satisfied if
    \begin{equation} \label{eq:Mdef}
	M = \left( \frac{n^{1-\epsilon}}{24 C_\epsilon} \right)^{\lambda/(1+\epsilon)} \muquant^{-\lambda}(\lambda)
	.
    \end{equation}
    From~\eqref{eq:nMcond}, we have that $(m \muquant^{-1}(\lambda))^{\lambda} > \gamma_\emptyset^{-1}$, and from~\eqref{eq:HCbound2} and $i \le |A_{d, \alpha, \bsgamma}(\ralpha(\bsh_i))|$ we have
    \[
	\ralpha(\bsh_m) \ge (m \muquant^{-1}(\lambda))^\lambda > \gamma_\emptyset^{-1} = \ralpha(\bszero)
	,
    \]
    and so ${\bszero \in \{ \bsh_i \st i \le m \}}$.
    The proof of Theorem~\ref{thm:upperbound} with $\Cinfty = 1$ and $r = 1$, implies that
    \begin{align*}
	\eworKor{d}(\lsqzeta)
	\le
	\frac{1}{M} + \sqrt{\frac{2}{n}\sum_{\substack{\bsh \in \Z^d \\ \ralpha(\bsh) > M}} \ralpha^{-2}(\bsh) + \sqrt{\frac{24 C_\epsilon}{n^{1-\epsilon}} \sum_{\substack{\bsh \in \Z^d \\ \ralpha(\bsh) > M}} \ralpha^{-4}(\bsh) \, \| \bsh_i \|_{\infty}^\epsilon}}
	,
    \end{align*}
    for a third of the possible choices for $\bszeta$.
    We now fix $\lambda \in (1/2, \alpha)$ arbitrarily.
    The subsequent bounds depend on this $\lambda$, and since they hold for all values, the final result holds for all $\lambda \in (1/2, \alpha)$.
    We proceed by using the fact that $\| \bsh \|_{\infty}^\alpha \le \ralpha(\bsh)$, leading to the bound
    \begin{align*}
	\eworKor{d}(\lsqzeta)
	&\le
	\frac{1}{M} + \sqrt{\frac{2}{n \, M^{2 - 1/\lambda}} \muquant(\lambda) + \sqrt{\frac{24 C_\epsilon}{n^{1-\epsilon}} \sum_{\substack{\bsh \in \Z^d \\ \ralpha(\bsh) > M}} \ralpha^{\epsilon/\alpha-4}(\bsh) }}
	\\
	&\le
	\frac{1}{M} + \sqrt{\frac{2}{n \, M^{2 - 1/\lambda}} \muquant(\lambda) + \sqrt{\frac{24 C_\epsilon}{n^{1-\epsilon} M^{4-\epsilon/\alpha-1/\lambda}} \muquant(\lambda) }}
	.
    \end{align*}
    Substituting the value of $M$ from~\eqref{eq:Mdef}, we obtain
    \begin{align*}
	&\eworKor{d}(\lsqzeta)
	\\*
	&\hspace{5mm}\le
	\left( \frac{n^{1-\epsilon}}{24 C_\epsilon} \right)^{-\lambda/(1+\epsilon)} \muquant^{\lambda}(\lambda)
	\\*
	&\hspace{10mm}+
	\sqrt{
	    2 (24 C_\epsilon)^{\frac{2\lambda - 1}{1+\epsilon}} n^{-1-\frac{1-\epsilon}{1+\epsilon}(2\lambda-1)} \muquant^{2 \lambda}(\lambda)
	    +
	    (24 C_\epsilon)^{\frac{2 \lambda - \frac{\epsilon \lambda}{2 \alpha} + \frac{\epsilon}{2}}{1 + \epsilon}} n^{-\frac{1-\epsilon}{1+\epsilon} (2 \lambda - \frac{\epsilon \lambda}{2 \alpha} + \frac{\epsilon}{2})} \muquant^{2\lambda - \frac{\epsilon \lambda}{2 \alpha}}(\lambda)
	}
	\\
	&\hspace{5mm}\le
	\left( (24 C_\epsilon)^{\frac{\lambda}{1+\epsilon}} n^{-\lambda \frac{1-\epsilon}{1+\epsilon}} + \sqrt{3} (24 C_\epsilon)^{\frac{\lambda - \frac{\epsilon \lambda}{4 \alpha} + \frac{\epsilon}{4}}{1+\epsilon}} n^{- \frac{1-\epsilon}{1+\epsilon} \lambda} \right) \muquant^\lambda(\lambda)
	\\
	&\hspace{5mm}\le
	3 (24 C_\epsilon)^{\frac{\lambda - \frac{\epsilon \lambda}{4 \alpha} + \frac{\epsilon}{4}}{1+\epsilon}} n^{-\lambda \frac{1-\epsilon}{1+\epsilon}} \muquant^\lambda(\lambda)
	\\
	&\hspace{5mm}\le
	3 (24 C_\epsilon)^{\lambda} n^{-\lambda \frac{1-\epsilon}{1+\epsilon}} \muquant^\lambda(\lambda)
	,
    \end{align*}
    where we have used the fact that $-\frac{\epsilon \lambda}{2 \alpha} + \frac{\epsilon}{2} > 0$ and that $\frac{1}{1+\epsilon} (\lambda - \frac{\epsilon \lambda}{4 \alpha} + \frac{\epsilon}{4}) \le \lambda$.
\end{proof}

We can also develop the bound on Theorem~\ref{thm:rationalbound}.

\begin{corollary} \label{cor:korobovrational}
    Let $d \in \Npos$, $\alpha > 1/2$, and $\bsgamma = \{ \gamma_\fraku \}_{\fraku \subseteq \{ 1, \ldots, d \}}$ be a set of weights with $\gamma_\fraku \in (0, 1]$.
    For the Korobov space, $\KorSpace$,
    if $n \in \Npos$ is large enough such that~\eqref{eq:nMcond} holds, i.e.,
    \begin{align*}
	\left| A_{d, \alpha, \bsgamma}\left( M \right) \right|^\lambda \muquant^{-\lambda}(\lambda) > \gamma_\emptyset^{-1}
	\quad
	\text{where}
	\quad
	M = \left( \frac{n^{1-\epsilon}}{24 C_\epsilon} \right)^{\lambda/(1+\epsilon)} \muquant^{-\lambda}(\lambda)
	,
    \end{align*}
    and $N \in \Npos$ is a prime, satisfying
    \begin{equation} \label{eq:Nbound}
        N
        \ge
        2 n^{2 + \frac{1}{2\alpha - \alpha/\lambda}}
        ,
    \end{equation}
    then, for $m = |A_{d, \alpha, \bsgamma}(M)|$, there exists a generating vector $\bsz \in \{ 1, \ldots, N \}^d$ for which we have the following bound
    \[
	\eworKor{d}(\lsqNz)
	\le
	4 \left( 24 C_\epsilon \, \muquant(\lambda) \, n^{-\frac{1-\epsilon}{1+\epsilon}} \right)^\lambda
	,
    \]
    with $0 < \epsilon \le 1$ and $1/2 < \lambda < \alpha$, and $\muquant$ defined by~\eqref{eq:mu_quantity}.
    Furthermore, at least a third of the possible choices of $\bsz$ will produce this bound.
    If we extend $\{ \gamma_\fraku \}_{\fraku \subseteq \Npos}$ to a system of weights for all finite subsets of $\Npos$, and if
    \[
       \sum_{\substack{\fraku \subset \Npos \\ |\fraku| < \infty}} \gamma_\fraku^{1/\alpha} \big( 2 \zeta(\alpha/\lambda) \big)^{|\fraku|}
       <
       \infty
    \]
    holds, then the bound can be stated with an absolute constant independent of the dimension.
\end{corollary}

\begin{proof}
    The first part of the proof is identical to the proof of Corollary~\ref{cor:korobovbound}.
    We have the following
    \[
	\left\| f - \lsqNz\!(f) \right\|_{L_2}
	\le
	\left( \frac{1}{M} + \left\| \Phinzm^+ \right\|_{2 \to 2} \left\| \PhinNzinf \Dlambdam \right\|_{2 \to 2} \right) \big\| f \big\|_{d, \alpha, \bsgamma}
	.
    \]

    Recall that selecting
    \[
	M = \left( \frac{n^{1-\epsilon}}{24 C_\epsilon} \right)^{\lambda/(1+\epsilon)} \muquant^{-\lambda}(\lambda)
	,
    \]
    implies that~\eqref{eq:kormbound} holds and so~\eqref{eq:firstNbound} holds using~\eqref{eq:Nbound}.
    Also, by~\eqref{eq:nMcond} and~\eqref{eq:HCbound2}, we have ${\bszero \in \{ \bsh_i \st i \le m \}}$.
    The proof of Theorem~\ref{thm:rationalbound} with $\Cinfty = 1$ and $r = 1$, implies that
    \begin{align*}
	&\left\| \Phinzm^+ \right\|_{2 \to 2}^2 \left\| \PhinNzinf \Dlambdam \right\|_{2 \to 2}^2
	\\*
	&\hspace{10mm}\le
	2 \sum_{\substack{i > m \\ \bsh_i \equiv_N \bszero}} \ralpha^{-2}(\bsh_i) + \frac{2}{n} \sum_{i > m} \ralpha^{-2}(\bsh)
	\\
	&\hspace{15mm}+
	\vast( \frac{24 C_\epsilon}{n^{1-\epsilon}} \sum_{\substack{i > m \\ \bsh_i \not\equiv_N \bszero}} \| \bsh_i \|_{\infty}^{\epsilon} \ralpha^{-4}(\bsh_i)  + \frac{24}{n} \sum_{\substack{i, j > m \\ \bsh_i \not\equiv_N \bszero \\ \bsh_j \not\equiv_N \bszero \\ \| \bsh_j \|_{\infty} > N/(2 n)}} \ralpha^{-2}(\bsh_i) \, \ralpha^{-2}(\bsh_j) \vast)^{1/2}
	,
    \end{align*}
    for a third of the possible choices of $\bsz$.

    As with the previous corollary, we can arbitrarily fix $\lambda \in (1/2, \alpha)$.
    A bound on the first two sums can be derived as follows
    \begin{align*}
	2 \sum_{\substack{i > m \\ \bsh_i \equiv_N \bszero}} \ralpha^{-2}(\bsh_i) + \frac{2}{n} \sum_{i > m} \ralpha^{-2}(\bsh_i)
	&\le
	\frac{2}{N^{2\alpha}} \sum_{i = 1}^\infty \ralpha^{-2}(\bsh_i) + \frac{2}{n M^{2-1/\lambda}} \mu(\lambda)
	\\
	&\le
	\frac{4}{n M^{2-1/\lambda}} \mu(\lambda)
	\\
	&\le
	4 (24 C_\epsilon)^{\frac{2\lambda-1}{1+\epsilon}} n^{\frac{1-\epsilon}{1+\epsilon} (1 - 2\lambda) - 1} \mu^{2\lambda}(\lambda)
	.
    \end{align*}
    In the second inequality, we use the fact that the bound on the first term is smaller than the second.

    We can continue with the terms inside the square root.
    The first can be bound as follows
    \begin{align*}
	&\frac{24 C_\epsilon}{n^{1-\epsilon}} \sum_{\substack{i > m \\ \bsh_i \not\equiv_N \bszero}} \| \bsh_i \|_{\infty}^{\epsilon} \, \ralpha^{-4}(\bsh_i)
	\\*
	&\hspace{10mm}\le
	\frac{24 C_\epsilon}{n^{1-\epsilon}} \sum_{\substack{i > m \\ \bsh_i \not\equiv_N \bszero}} \ralpha^{\epsilon/\alpha-4}(\bsh_i)
	\\
	&\hspace{10mm}\le
	\frac{24 C_\epsilon}{n^{1-\epsilon}} \frac{1}{M^{4-1/\lambda-\epsilon/\alpha}} \mu(\lambda)
	\\
	&\hspace{10mm}\le
	(24 C_\epsilon)^{\frac{1}{1+\epsilon}(4\lambda + \epsilon - \frac{\epsilon\lambda}{\alpha})} n^{\frac{1-\epsilon}{1+\epsilon}(\frac{\epsilon\lambda}{\alpha} - \epsilon - 4\lambda)}  \mu^{4\lambda - \frac{\epsilon\lambda}{\alpha}}(\lambda)
	\label{eq:firststar}
	\tag{$*$}
	.
    \end{align*}
    For the second term, since $\| \bsh \|_{\infty}^\alpha \le \ralpha(\bsh)$, defining $S := \left( \frac{N}{2n} \right)^\alpha$, we obtain
    \begin{align*}
	&\frac{24}{n} \sum_{\substack{i, j > m \\ \bsh_i \not\equiv_N \bszero \\ \bsh_j \not\equiv_N \bszero \\ \| \bsh_j \|_{\infty} > N/(2 n)}} \ralpha^{-2}(\bsh_i) \, \ralpha^{-2}(\bsh_j)
	\\*
	&\hspace{15mm}\le
	\frac{24}{n} \sum_{\substack{i, j > m \\ \ralpha(\bsh_j) > S \\ \bsh_i \not\equiv_N \bszero \\ \bsh_j \not\equiv_N \bszero}} \ralpha^{-2}(\bsh_i) \, \ralpha^{-2}(\bsh_j)
	\\
	&\hspace{15mm}\le
	\frac{24}{n} \left( \frac{1}{S M} \right)^{2-1/\lambda} \mu^2(\lambda)
	\\
	&\hspace{15mm}\le
	\frac{24}{n} \left( \frac{2n}{N} \right)^{2\alpha-\alpha/\lambda}  \frac{1}{M^{2-1/\lambda}}  \mu^2(\lambda)
	\\
	&\hspace{15mm}\le
	24 \times 2^{2\alpha-\alpha/\lambda} (24 C_\epsilon)^{\frac{2 \lambda - 1}{1+\epsilon}} n^{\frac{1-\epsilon}{1+\epsilon}(1 - 2\lambda) - 1 + 2\alpha - \alpha/\lambda} N^{\alpha/\lambda - 2\alpha} \mu^{2\lambda+1}(\lambda)
	.
	\label{eq:secondstar}
	\tag{$**$}
    \end{align*}
    The bound on $N$ is chosen so that the second bound~\eqref{eq:secondstar} is less than the first~\eqref{eq:firststar}, this means that
    \begin{align*}
	&\left\| \Phinzm^+ \right\|_{2 \to 2}^2 \left\| \PhinNzinf \Dlambdam \right\|_{2 \to 2}^2
	\\[1mm]
	&\hspace{5mm}\le
	4 (24 C_\epsilon)^{\frac{2\lambda-1}{1+\epsilon}} n^{\frac{1-\epsilon}{1+\epsilon} (1 - 2\lambda) - 1} \mu^{2\lambda}(\lambda)
	+
	\sqrt{2 (24 C_\epsilon)^{\frac{1}{1+\epsilon}(4\lambda + \epsilon - \frac{\epsilon\lambda}{\alpha})} \, n^{\frac{1-\epsilon}{1+\epsilon}(\frac{\epsilon\lambda}{\alpha} - \epsilon - 4\lambda)}  \mu^{4\lambda - \frac{\epsilon\lambda}{\alpha}}(\lambda)}
	\\[1mm]
	&\hspace{5mm}\le
	6 (24 C_\epsilon)^{\frac{1}{1+\epsilon}(2\lambda + \frac{\epsilon}{2} - \frac{\epsilon\lambda}{2\alpha})} \, n^{-2 \frac{1-\epsilon}{1+\epsilon} \lambda} \mu^{2\lambda}(\lambda)
	.
    \end{align*}
    This can then be applied to the worst-case error.
    For at least $1/3$ of the possibilities of $\bsz$, we have
    \begin{align*}
	\eworKor{d}(\lsqNz)
	&\le
	\frac{1}{M} + \left\| \Phinzm^+ \right\|_{2 \to 2} \left\| \PhinNzinf \Dlambdam \right\|_{2 \to 2}
	\\[1mm]
	&\le
	\left( \frac{n^{1-\epsilon}}{24 C_\epsilon} \right)^{-\lambda/(1+\epsilon)} \muquant^{\lambda}(\lambda)
	+
	\sqrt{6 \, (24 C_\epsilon)^{\frac{1}{1+\epsilon}(2\lambda + \frac{\epsilon}{2} - \frac{\epsilon\lambda}{2\alpha})} \, n^{-2 \frac{1-\epsilon}{1+\epsilon} \lambda} \mu^{2\lambda}(\lambda)}
	\\[1mm]
	&\le
	4 \, (24 C_\epsilon)^{\frac{1}{1+\epsilon}(\lambda + \frac{\epsilon}{4} - \frac{\epsilon\lambda}{4\alpha})} n^{-\frac{1-\epsilon}{1+\epsilon}\lambda} \mu^{\lambda}(\lambda)
	\\[1mm]
	&\le
	4 \, (24 C_\epsilon)^{\lambda} n^{-\frac{1-\epsilon}{1+\epsilon}\lambda} \mu^{\lambda}(\lambda)
	,
    \end{align*}
    where we have again used the fact that $-\frac{\epsilon \lambda}{2 \alpha} + \frac{\epsilon}{2} > 0$ and that $\frac{1}{1+\epsilon} (\lambda - \frac{\epsilon \lambda}{4 \alpha} + \frac{\epsilon}{4}) \le \lambda$.
\end{proof}

\bibliography{Bibliography}
\bibliographystyle{abbrv}

\end{document}

%% file: macros.tex
\usepackage{booktabs}
\usepackage[english]{babel}
\usepackage[utf8x]{inputenc}
\usepackage[a4paper]{geometry}
\usepackage{lmodern}
\usepackage{abstract}
\usepackage{float}
\usepackage{xcolor}
\usepackage{comment}
\usepackage{tikz}
\usepackage{pgfplots}
\usepackage{pgfplotstable}
\usepackage{readarray}
\usepackage{subcaption}

\usepackage{graphicx}
\usepackage[T1]{fontenc}
\usepackage{microtype}
\usepackage{cite}
\usepackage{array,colortbl}
\usepackage{amsmath,amsthm,mathtools}
\usepackage{amsfonts,bm,amssymb}
\usepackage{algorithm}
\usepackage{algorithmic}

\usepackage[hidelinks,hypertexnames=false,hyperindex=true,pdfpagelabels,linktoc=all]{hyperref}
\usepackage{bookmark}
\usepackage{enumitem}

\DeclareMathOperator{\supp}{supp}

\newcommand{\st}{\: : \:}

\makeatletter
\newcommand{\vast}{\bBigg@{4}}
\newcommand{\Vast}{\bBigg@{5}}
\makeatother

\allowdisplaybreaks

\DeclareFontFamily{U}{mathx}{\hyphenchar\font45}
\DeclareFontShape{U}{mathx}{m}{n}{
<5> <6> <7> <8> <9> <10>
<10.95> <12> <14.4> <17.28> <20.74> <24.88>
mathx10
}{}
\DeclareSymbolFont{mathx}{U}{mathx}{m}{n}
\DeclareFontSubstitution{U}{mathx}{m}{n}
\DeclareMathAccent{\widecheck}{0}{mathx}{"71}
\def\citep#1#2{\cite[{#1}]{#2}}
\newcommand{\bbE}{{\mathbb{E}}}

\newcommand{\bbT}{{\mathbb{T}}}

\newcommand{\N}{{\mathbb{N}}} 
\newcommand{\R}{{\mathbb{R}}} 
\newcommand{\Z}{{\mathbb{Z}}} 
\newcommand{\CC}{{\mathbb{C}}} 

\DeclareSymbolFont{bbold}{U}{bbold}{m}{n}
\DeclareSymbolFontAlphabet{\mathbbold}{bbold}

\newcommand{\bbone}{{\mathbbold{1}}}

\newcommand{\bsf}{{\boldsymbol{f}}}

\newcommand{\bsh}{{\boldsymbol{h}}}

\newcommand{\bst}{{\boldsymbol{t}}}

\newcommand{\bsx}{{\boldsymbol{x}}}
\newcommand{\bsy}{{\boldsymbol{y}}}
\newcommand{\bsz}{{\boldsymbol{z}}}

\newcommand{\bszero}{{\boldsymbol{0}}} 

\newcommand{\bsgamma}{{\boldsymbol{\gamma}}}

\newcommand{\bszeta}{{\boldsymbol{\zeta}}}

\newcommand{\bssigma}{{\boldsymbol{\sigma}}}

\newcommand{\calH}{{\mathcal{H}}}
\newcommand{\calI}{{\mathcal{I}}}

\newcommand{\calP}{{\mathcal{P}}}

\newcommand{\fraku}{{\mathfrak{u}}}

\newcommand{\rme}{{\mathrm{e}}}

\newcommand{\rmi}{{\mathrm{i}}}

\newcommand{\floor}[1]{\left\lfloor #1 \right\rfloor} 
\newcommand{\rd}{\,\mathrm{d}} 

\DeclareMathOperator{\var}{Var}

\makeatletter 
  \providecommand*{\toclevel@author}{999}
  \providecommand*{\toclevel@title}{0}
\makeatother

\theoremstyle{plain}
  \newtheorem{theorem}{Theorem}
  
  \newtheorem{lemma}{Lemma}
  \newtheorem{corollary}{Corollary}
  
\theoremstyle{definition}

\theoremstyle{remark}
  \newtheorem{remark}{Remark}

%% file: kronecker-lattices.bbl
\begin{thebibliography}{10}

\bibitem{BKPU2024}
F.~Bartel, L.~Kämmerer, D.~Potts, and T.~Ullrich.
\newblock On the reconstruction of functions from values at subsampled
  quadrature points.
\newblock {\em Mathematics of Computation}, 93:785--809, 2024.

\bibitem{CKNS2020}
R.~Cools, F.~Y. Kuo, D.~Nuyens, and I.~H. Sloan.
\newblock Lattice algorithms for multivariate approximation in periodic spaces
  with general weight parameters.
\newblock In S.~C. Brenner, I.~E. Shparlinski, C.~Shu, and D.~B. Szyld,
  editors, {\em 75 years of Mathematics of Computation}, pages 93--113.
  American Mathematical Society, 2020.

\bibitem{DKU2023}
M.~Dolbeault, D.~Krieg, and M.~Ullrich.
\newblock A sharp upper bound for sampling numbers in {$L_2$}.
\newblock {\em Applied and Computational Harmonic Analysis}, 63:113--134, 2023.

\bibitem{DTU2018}
D.~Dũng, V.~Temlyakov, and U.~Tino.
\newblock {\em Hyperbolic Cross Approximation}.
\newblock Advanced Courses in Mathematics - CRM Barcelona. Birkhäuser Cham,
  2018.

\bibitem{HW1960}
G.~Hardy and E.~Wright.
\newblock {\em An Introduction to the Theory of Numbers}.
\newblock Oxford mathematics. Oxford at the Clarendon Press, fourth edition,
  1960.

\bibitem{KKKNS2022}
V.~Kaarnioja, Y.~Kazashi, F.~Y. Kuo, F.~Nobile, and I.~H. Sloan.
\newblock Fast approximation by periodic kernel-based lattice-point
  interpolation with application in uncertainty quantification.
\newblock {\em Numerische Mathematik}, 150(1):33--77, 2022.

\bibitem{K2013}
L.~K{\"a}mmerer.
\newblock Reconstructing multivariate trigonometric polynomials by sampling
  along generated sets.
\newblock In J.~Dick, F.~Y. Kuo, G.~W. Peters, and I.~H. Sloan, editors, {\em
  Monte Carlo and Quasi-Monte Carlo Methods 2012}, pages 439--454, Berlin,
  Heidelberg, 2013. Springer Berlin Heidelberg.

\bibitem{KU2021a}
D.~Krieg and M.~Ullrich.
\newblock Function values are enough for {$L_2$}-approximation.
\newblock {\em Foundations of Computational Mathematics}, 21:1141--1151, 2021.

\bibitem{KU2021b}
D.~Krieg and M.~Ullrich.
\newblock Function values are enough for {$L_2$}-approximation: Part {II}.
\newblock {\em Journal of Complexity}, 66:101569, 2021.

\bibitem{KN1974}
L.~Kuipers and H.~Niederreiter.
\newblock {\em Uniform Distribution of Sequences}.
\newblock Wiley, 1974.

\bibitem{KMN2024}
F.~Y. Kuo, W.~Mo, and D.~Nuyens.
\newblock Constructing embedded lattice-based algorithms for multivariate
  function approximation with a composite number of points.
\newblock {\em Constructive Approximation}, 2024.

\bibitem{K2014}
L.~Kämmerer.
\newblock {\em High Dimensional Fast Fourier Transform Based on Rank-1 Lattice
  Sampling}.
\newblock PhD thesis, Technische Universität Chemnitz, 2014.

\bibitem{KPV2015}
L.~Kämmerer, D.~Potts, and T.~Volkmer.
\newblock Approximation of multivariate periodic functions by trigonometric
  polynomials based on rank-1 lattice sampling.
\newblock {\em Journal of Complexity}, 31(4):543--576, 2015.

\bibitem{KUV2021}
L.~Kämmerer, T.~Ullrich, and T.~Volkmer.
\newblock Worst-case recovery guarantees for least squares approximation using
  random samples.
\newblock {\em Constructive Approximation}, 54:295--352, 2021.

\bibitem{L1988}
G.~Larcher.
\newblock On the distribution of $s$-dimensional {K}ronecker-sequences.
\newblock {\em Acta Arithmetica}, 51(4):335--347, 1988.

\bibitem{NW2008}
E.~Novak and H.~Wo{\'z}niakowski.
\newblock {\em Tractability of Multivariate Problems: Volume I: Linear
  information}.
\newblock EMS Tracts in Mathematics. European Mathematical Society, 2008.

\bibitem{NW2010}
E.~Novak and H.~Wo{\'z}niakowski.
\newblock {\em Tractability of Multivariate Problems: Volume II: Standard
  Information for Functionals}.
\newblock EMS Tracts in Mathematics. European Mathematical Society, 2010.

\bibitem{NW2012}
E.~Novak and H.~Wo{\'z}niakowski.
\newblock {\em Tractability of Multivariate Problems, Volume III: Standard
  Information for Operators}.
\newblock EMS Tracts in Mathematics. European Mathematical Society, 2012.

\bibitem{O1982}
A.~M. Ostrowski.
\newblock On the error term in multidimensional diophantine approximation.
\newblock {\em Acta Arithmetica}, 41(2):163--183, 1982.

\bibitem{TW1980}
J.~F. Traub and H.~Woźniakowski.
\newblock {\em A General Theory of Optimal Algorithms}.
\newblock ACM monograph series. Academic press, New York (N.Y.), 1980.

\bibitem{CVW2007}
B.~Vandewoestyne, R.~Cools, and T.~Warnock.
\newblock On obtaining quadratic and cubic error convergence using weighted
  {K}ronecker-sequences.
\newblock {\em Computing}, 80:75--94, 2007.

\bibitem{W1916}
H.~Weyl.
\newblock {Ü}ber die {G}leichverteilung von {Z}ahlen mod. {E}ins.
\newblock {\em Mathematische Annalen}, 77:313--352, 1916.

\end{thebibliography}
